\documentclass[reqno,12pt]{amsart} 
\usepackage{amsmath}
\usepackage{amssymb}
\usepackage{amsthm}
\newcommand\NoBlackBoxes{\global\overfullrule0pt}
\NoBlackBoxes
\theoremstyle{plain} 
\newtheorem{theorem}{Theorem} 
\newtheorem{lemma}[theorem]{Lemma}
\newtheorem{proposition}[theorem]{Proposition}

\newtheorem{corollary}[theorem]{Corollary}
\def\4{\kern1pt}

\def\6{\vphantom0}

\def\8{\kern-10pt}
\def\7#1{_{(#1)}}

\theoremstyle{definition}
\newtheorem*{assumption (H1)}{Assumption (H1)}
\newtheorem*{assumption (H2)}{Assumption (H2)}

\theoremstyle{remark}
\newtheorem{remark}[theorem]{Remark}

\numberwithin{equation}{section}
\numberwithin{theorem}{section}
\makeatletter
\let\serieslogo@\relax
\let\@setcopyright\relax

\def\speciallabelmark#1{\def\@currentlabel{#1}}
\makeatother

\newcommand{\Cal}{\mathcal}

\newcommand{\beqa}{\begin{eqnarray}}
\newcommand{\beqan}{\begin{eqnarray*}}
\newcommand{\eeqa}{\end{eqnarray}}
\newcommand{\eeqan}{\end{eqnarray*}}
\def\beq#1\eeq{\begin{equation}#1\end{equation}}

\begin{document}

\def\ffrac#1#2{\raise.5pt\hbox{\small$\4\displaystyle\frac{\,#1\,}{\,#2\,}\4$}}
\def\ovln#1{\,{\overline{\!#1}}}
\def\ve{\varepsilon}
\def\kar{\beta_r}

\title{Asymptotic Expansions in the~CLT in Free Probability}

\author{G. P. Chistyakov$^{1,2,3}$}
\thanks{1) Faculty of Mathematics , 
University of Bielefeld, Germany.} 
\thanks{2) Institute for Low Temperature Physics and Engineering, 
Kharkov, Ukraine}
\thanks{3) Research supported by SFB 701.}
\address
{Gennadii Chistyakov \newline
Fakult\"at f\"ur Mathematik\newline
Universit\"at Bielefeld\newline
Postfach 100131\newline
33501 Bielefeld \newline
Germany}
\email {chistyak@math.uni-bielefeld.de} 

\author{F. G\"otze$^{1,3}$}
\address
{Friedrich G\"otze\newline
Fakult\"at f\"ur Mathematik\newline
Universit\"at Bielefeld\newline
Postfach 100131\newline
33501 Bielefeld \newline
Germany}
\email {goetze@math.uni-bielefeld.de}

\date{December, 2011}


\maketitle
\markboth{ G. P. Chistyakov and F. G\"otze}{Edgeworth's Expansions
in Free Probability Theory}

\begin{abstract}
We prove Edgeworth type expansions for distribution functions of sums
of free random variables under minimal moment conditions. The proofs are based on
the analytic definition of free convolution. We apply these results to the expansion of densities to derive
expansions for the free entropic distance of sums to the Wigner law.



\end{abstract}

\section{Introduction}

In recent years a number of papers are devoted to limit theorems
for the free convolutions of probability measures. Free convolutions were introduced
by D. Voiculescu~\cite{Vo:1986}, \cite{Vo:1987}.
The~key concept is the~notion of freeness,
which can be interpreted as a~kind of independence for
noncommutative random variables. As in the~classical probability where
the~concept of independence gives rise to the~classical convolution, 
the~concept of freeness leads to a~binary operation on the~probability measures, 
the~free convolution. Many classical results 
in the~theory of addition of independent random variables have their
counterparts in Free Probability, such as the~Law of Large Numbers,
the~Central Limit Theorem, the~L\'evy-Khintchine formula and others.
We refer to Voiculescu, Dykema and Nica \cite{Vo:1992} and Hiai 
and Petz~\cite{HiPe:2000} for an~introduction 
to these topics. Bercovici and Pata \cite{BeP:1999} established
the~distributional behavior of sums of free identically distributed 
random variables and described explicitly the~correspondence 
between limit laws for free and classical additive convolution. 
Using subordination functions for the~definition of the~additive 
free convolution, 
Chistyakov and G\"otze \cite{ChG:2005a} generalized the~results of 
Bercovici and Pata to the~case of free
non-identically distributed random variables. It was shown that 
the~parallelism found by Bercovici and Pata holds in the~general case of
free non-identically distributed random variables (see \cite{BeW:2006} 
as well). This approach 
allowed us to obtain estimates of the~rate
of convergence of distribution functions of free sums. An analog of the~Berry-Esseen
inequality was proved for the~semicircle approximation in~\cite{ChG:2005a}. 
For related results see~\cite{Ka:2007}.

In this paper we 
obtain an~analogue of Edgeworth expansion in the~Central Limit Theorem 
(CLT for short) for free identically distributed random variables, based on the~method
of subordination functions.
In addition we shall give a~bound for the~remainder term in this 
expansion. In order to deduce this expansion we establish an approximation
of distribution of normalized sums of free random variables by the~free Meixner distributions.
An interesting feature of our expansions is that they hold without the not moment related
assumptions used in the~classical probability.
For bounded free random variables we give an asymptotic expansion for the density
of normalized sums. With the help of this expansion we obtain the rate of convergence
in the entropic free CLT.

The~paper is organized as follows. In Section~2 we formulate and
discuss the main results of the~paper. In Section~3 and 4 we formulate
auxiliary results. In Section~5 we describe a~formal expansion in the~Free
CLT and in Sections~6 and 7 we prove Edgeworth's expansion 
in the~CLT for free identically distributed random variables. Since the~proofs
of Theorem~\ref{2.1*th} and Theorem~\ref{2.1th} (see Section~2) are similar, we give a~proof of 
Theorem~\ref{2.1th} in details in Section~5 and an~outline of the~proof of Theorem~\ref{2.1*th} in Appendix.
In Section 8 we obtain the local CLT for free bounded identically distributed random variables.
In Section 9 we obtain the rate of convergence in the free CLT in total variation metric 
and in Section 10 we derive the asymptotic expansion for the free entropy of normalized sums.

\section{Results}
Denote by $\mathcal M$ the~family of all Borel probability measures
defined on the~real line $\mathbb R$. Define on $\mathcal M$
the~compositions laws denoted $*$ and $\boxplus$ as follows.
For $\mu,\nu\in\mathcal M$, let $\mu*\nu$ denote
the~classical convolution of $\mu$ and $\nu$. 
In probabilistic terms, $\mu*\nu=\mathcal L(X+Y)$, where $X$ and $Y$ are
independent random variables with $\mu=\mathcal L(X)$ and $\nu=\mathcal L(Y)$, 
respectively. Let $\mu\boxplus\nu$ be
the~free (additive) convolution of $\mu$ and $\nu$ introduced by 
Voiculescu~\cite{Vo:1986} for compactly supported measures. 
Free convolution was extended by
Maassen~\cite{Ma:1992} to measures with finite variance and by Bercovici
and Voiculescu~\cite{BeVo:1993} to the~class $\mathcal M$.
Thus, $\mu\boxplus\nu=\mathcal L(X+Y)$,
where $X$ and $Y$ are free random variables such that $\mu=\mathcal L(X)$ and $\nu=\mathcal L(Y)$.
There are free analogues of multiplicative
convolutions as well; these were first studied in Voiculescu~\cite{Vo:1987}.

Henceforth $X,X_1,X_2,\dots$ stands for a~sequence of identically distributed random variables with
distribution $\mu=\mathcal L(X)$. Define 
\begin{equation}\notag
m_k:=\int_{\mathbb R}u^k\,\mu(du)\quad\text{and}\quad
\beta_q:=\int_{\mathbb R}|u|^q\,\mu(du),
\end{equation}
where $k=0,1,\dots$ and $q>0$.

The~classical CLT says that if $X_1,X_2,\dots$ are independent
and identically distributed random variables with a~probability 
distribution $\mu$ such that $m_1=0$ and $m_2=1$, 
then the~distribution function $F_n(x)$ of 
\begin{equation}\label{2.1}
Y_n:=\frac{X_1+X_2+\dots +X_n}{\sqrt n}
\end{equation}
tends to the~standard Gaussian law $\Phi(x)$ as $n\to\infty$ uniformly
in $x$.

A~free analogue of this classical result was proved by Voiculescu~\cite{Vo:1985} for bounded
free random variables and 
later generalized by Maassen~\cite {Ma:1992} to unbounded random variables.
Other generalizations can be found in \cite{BeVo:1995}, \cite{BeP:1999}, \cite{ChG:2005a}, 
\cite {Ka:2007}--\cite {Ka:2007b}, \cite{P:1996}, \cite{Vo:2000}, \cite{Wa:2010}. 
When the~assumption of independence is 
replaced by the~freeness of the~noncommutative random variables
$X_1,X_2,\dots,X_n$, the~limit distribution function of (\ref{2.1}) is 
the~semicircle law $w(x)$, i.e., the~distribution function 
with the~density $p_w(x):=\frac 1{2\pi}\sqrt{(4-x^2)_+}, \,x\in\mathbb R$, 
where $a_+:=\max\{a,0\}$ for $a\in\mathbb R$. Denote by $\mu_w$ the~probability measure with
the~distribution function $w(x)$.

Write $\varphi(x):=\frac 1{\sqrt{2\pi}}e^{-x^2/2}$ and denote by
$H_m(x):=(-1)^m e^{x^2/2}\frac{d^m}{dx^m}e^{-x^2/2}$ the~Hermite polynomial
of degree $m$.

Assume that the~random variables $X_j$ are independent and have moments of all orders.
For the~distribution function $F_n(x)$ of $Y_n$ there exists a~formal expansion in
a~power series in $1/\sqrt n$ (see~\cite{GKo:1968}, \cite{Pe:1975}):
\begin{equation}\label{2.1a}
F_n(x)=\Phi(x)+\varphi(x)\sum_{p=1}^{\infty}\frac{Q_{p}(x)}{n^{p/2}},
\end{equation}
where
$$
Q_{p}(x)=-\sum H_{p+2s-1}(x)\prod_{m=1}^{p}\frac 1{k_m!}
\Big(\frac{\gamma_{m+2}}{(m+2)!}\Big)^{k_m}
$$
and $\gamma_m$ is the~cumulant of order $m$ of random variable $X$.
In the~last equality the~summation on the~right-hand side is carried
out over all nonnegative integer solutions $(k_1,\dots,k_m)$ of
the~equations 
\begin{equation}
k_1+2k_2+\dots+p k_{p}=p\quad\text{and}\qquad s=k_1+\dots+k_{p}.
\end{equation}
Note that $Q_1(x)=-m_3H_2(x)/6$.

In terms of characteristic functions (\ref{2.1a}) has the~form
\begin{equation}\label{2.1aa}
\int_{-\infty}^{\infty}e^{itx}\,dF_n(x)=e^{-t^2/2}+\sum_{m=1}^{\infty}
\frac{P_m(t)}{n^{m/2}}e^{-t^2/2},
\end{equation}
where
$$
\int_{-\infty}^{\infty}e^{itx}\,dQ_m(x)=P_m(t)e^{-t^2/2}.
$$

Esseen~\cite{Es:1945} proved that if the~random variables $X_j$ are independent, non-lattice distributed
and $\beta_3<\infty$, 
then $F_n(x)$ admits the~following asymptotic expansion  
\begin{equation}\label{2.2}
F_n(x)=\Phi(x)-\frac{m_3}{6\sqrt{n}}H_2(x)\varphi(x)+o(1/\sqrt n)
\end{equation}
which holds uniformly in $x$.

If the~random variables $X_j$ are independent and are lattice distributed, that is
they take values in an~arithmetic
progression $\{a+kh;k=0,\pm 1,\dots\}$ ($h$ being maximal), and $\beta_3<\infty$, then
\begin{equation}\label{2.3}
F_n(x)=\Phi(x)+\frac 1{\sqrt{n}}\varphi(x)\Big(-\frac{m_3}{6}H_2(x)
+hT\Big(\frac{x\sqrt n}h-\frac{an}h\Big)\Big)+o(1/\sqrt n),
\end{equation}
uniformly in $x$, where $T(x):=[x]-x+1/2$. 

If absolute moments $\beta_k$ of order $k>3$ exist, then generalizations of the~asymptotic
expansions (\ref{2.2}) and (\ref{2.3}) hold under 
additional conditions on the~characteristic function of $X$ \cite{Pe:1975}.

An~analytical approach using subordination functions
allowed us to give explicit estimates for the~rate of convergence of 
distribution functions of $Y_n$ in the~case of free random variables. 
We demonstrated this (see 
\cite {ChG:2005a}) by proving
a~semicircle approximation theorem (an~analogue
of the~Berry-Esseen inequality \cite{Pe:1975}, p. 111). In this paper 
we shall establish Edgeworth expansion in the~semicircle approximation theorem and a~complete analogue
of the~Berry-Esseen inequality for identically distributed free random variables.

We now formulate the~main results of the~paper. As before we denote
by $F_n(x)$ the~distribution function of $Y_n$ where $X_j$ are free
random variables with the same distribution $\mu$. Assume 
as well that $X_j$ have moments of arbitrary order and 
$m_1=0,\,m_2=1$. We denote by $\mu_n$ the~distribution of $Y_n$.
Denote by $U_m(x)$ the~Chebyshev polynomial of   
the~second kind of degree $m$, i. e., 
$$
U_m(x)=U_m(\cos\theta)
:=\frac{\sin(m+1)\theta}{\sin\theta},\quad m=1,2,\dots.
$$
It is easy to see $U_1(x)=2x,\,U_2(x)=4x^2-1,\,U_3(x)=4x(2x^2-1)$.

It turns out that there exists an~analogue of the~formal
expansion (\ref{2.1aa}) for $F_n(x)$. To formulate it we need the~following notation.
Define the~Cauchy transform of $\mu\in\mathcal M$ by
\begin{equation}\label{2.3a}
G_{\mu}(z)=\int_{\mathbb R}\frac{\mu(dx)}{z-x},\qquad z\in\mathbb C^+,
\end{equation}
where $\mathbb C^+$ denotes the~open upper half of
the~complex plane.
The~formal expansion has the~form
\begin{equation}\label{2.3b}
G_{\mu_n}(z)=G_{\mu_w}(z)+
\sum_{k=1}^{\infty}\frac{B_k(G_{\mu_w}(z))}{n^{k/2}},
\end{equation}
where
\begin{equation}\label{2.3bb}
B_k(z)=\sum c_{p,m}\frac {z^p}{(1/z-z)^m}
\end{equation} 
with real coefficients $c_{p,m}$ which depend on 
the~ free cumulants $\alpha_3,\dots,\alpha_{k+2}$ 
and do not depend on $n$.
The~free cumulants will be defined in Section~3, (\ref{3.7}). Here we note
that $\alpha_3=m_3$ and $\alpha_4=m_4-2$.
The~summation on the~right-hand side of (\ref{2.3bb}) is taken
over a~finite set of non-negative integer pairs $(p,m)$. 
The~coefficients $c_{p,m}$ can be calculated explicitly. For the~cases
$k=1,2$ we have 
$$
B_1(z)=\alpha_3\frac {z^3}{1/z-z}
$$
and
\begin{align}\label{2.3bbb}
B_2(z)=
\big(\alpha_4-\alpha_3^2\big)\frac {z^4}{1/z-z}
+\alpha_3^2\Big(\frac{z^5}{(1/z-z)^2}+\frac{z^2}{(1/z-z)^3}\Big).
\end{align}

Note that
\begin{equation}\label{2.3c}
B_1(G_{\mu_w}(z))=\frac {\alpha_3}{\sqrt{z^2-4}}G_{\mu_w}^3(z)
=-\alpha_3\int_{-2}^2\frac 1{z-x}\,d\Big(\frac 13 U_2(x/2)
p_w(x)\Big),\quad z\in\mathbb C^+.
\end{equation}
If $\alpha_3=0$, then 
\begin{equation}\label{2.3d}
B_2(G_{\mu_w}(z))=\frac {\alpha_4}{\sqrt{z^2-4}}G_{\mu_w}^4(z)=
-\alpha_4\int_{-2}^2\frac 1{z-x}\,d\Big(\frac 14 U_3(x/2)
p_w(x)\Big),\quad z\in\mathbb C^+.
\end{equation}

Now we can formulate a~counterpart of Edgeworth expansion in the~Free CLT.
We obtain this counterpart from the~following results in which we establish
an~approximation of the~measures $\mu_n$ by the~free Meixner measures.
Consider the~three-parameter family of probability measures $\{\mu_{a,b,d}:
a\in\mathbb R, b<1, d<1\}$ with the~reciprocal Cauchy transform
\begin{equation}\label{2.3h}
\frac 1{G_{\mu_{a,b,d}}(z)}=a+\frac 12\Big((1+b)(z-a)+\sqrt{(1-b)^2(z-a)^2-4(1-d)}\Big),
\quad z\in\mathbb C,
\end{equation}
which we will call the~free centered Meixner measures (i.e. with mean zero). 
In this formula we choose the~branch of the~square root determined by the~condition
$\Im z>0$ implies $\Im (1/G_{\mu_{a,b,d}}(z))\ge 0$.
These measures are counterparts of the~classical measures discovered by Meixner~\cite{Me:1934}.
The~free Meixner type measures occurred in many places in the~literature, see
~\cite{An:2003}, \cite{BoBr:2006}, \cite{BoSp:1991}, \cite{BoLeSp:1996}, \cite{Ke:1959},
\cite{Mc:1981}, \cite{SaYo:2001}.

Assume that $m_4<\infty,m_1=0,m_2=1$ and denote 
\begin{equation}\label{2.3****}
a_n:=\frac{m_3}{\sqrt n},\quad b_n:=\frac{m_4-m_3^2-1}n,\quad
d_n:=\frac{m_4-m_3^2}n,\quad n\in\mathbb N.
\end{equation}
In the~sequel we will use the~free Meixner measures of the~form $\mu_{0,0,0}=w,\mu_{a_n,0,0}$ if 
$\beta_3<\infty,m_1=0,m_2=1$ 
and $\mu_{a_n,b_n,d_n}$ if $m_4<\infty,m_1=0,m_2=1$ and $n>m_4$. 

Recall that a~probability measure $\mu$ is $\boxplus$-infinitely divisible if for every $n\in\mathbb N$
there exists $\nu_n\in\mathcal M$ such that $\mu=\nu_n\boxplus\nu_n\boxplus\dots
\boxplus\nu_n$ ($n$ times).

Using the~results of Saitoh and Yoshida~\cite{SaYo:2001}, we will show in Section~4
that under the~assumptions
$\beta_3<\infty$ and $n\ge m_3^2$ the~free Meixner measure
$\mu_{a_n,0,0}$ is absolute continuous with a~density of the~form (\ref{2.3g}), where $a=a_n,b=0,d=0$,
and $\mu_{a_n,0,0}$ is $\boxplus$-infinitely divisible.
Under the~assumptions $m_4<\infty$ and $n\ge 3m_4$ the~free Meixner measure $\mu_{a_n,b_n,d_n}$ 
is absolute continuous with a~density of the~form (\ref{2.3g}), where $a=a_n,b=b_n,d=d_n$,
and $\mu_{a_n,b_n,d_n}$ is $\boxplus$-infinitely divisible.

We now introduce some further notations. Assume that $\beta_q<\infty$ for some $q\ge 2$. Introduce
the Lyapunov fractions
\begin{equation}\label{2.4*}
L_{qn}:=\frac {\beta_q}{n^{(q-2)/2}}\quad\text{and let}\quad
\rho_{q}(\mu,t):=\int_{|u|>t}|u|^{q}\,\mu(du),\,\, t>0.
\end{equation}
Write
$$
q_1:=\min\{q,3\},\quad   q_2:=\min\{q,4\}, 
\quad q_3:=\min\{q,5\}. 
$$ 
Then denote, for $n\in\mathbb N$, 
\begin{equation}\label{2.4**}
\eta_{qs}(n):=\inf_{0<\varepsilon\le 10^{-1/2}}g_{qns}(\varepsilon),\quad\text{where}\quad
g_{qns}(\varepsilon):=\varepsilon^{s+2-q_s}+\frac{\rho_{q_s}(\mu,\varepsilon\sqrt n)}{\beta_{q_s}}\varepsilon^{-q_s} 
\end{equation}
provided that $\beta_q<\infty,\,q\ge s+1$, for $s=1,2,3$, respectively.
It is easy to see that $0<\eta_{qs}(n)\le 10^{1+s/2}+1$ for $s+1\le q_s\le s+2$ and $\eta_{qs}(n)\to 0$
monotonically as $n\to\infty$ if $s+1\le q_s<s+2$, and $\eta_{qs}(n)\ge 1,\,n\in\mathbb N$, if $q_s=s+2$.

By agreement the~symbols $c,c_1,c_2,\dots$ and $c(\mu),c_1(\mu),c_2(\mu),\dots$ shall
denote absolute positive constants and positive constants depending on $\mu$ only, respectively. By~$c$ and $c(\mu)$
we denote generic constants in different (or even in the same) formulae. The symbols $c_1,c_2,\dots$
and $c_1(\mu),c_2(\mu),\dots$ are applied for  explicit constants.

\begin{theorem}\label{2.1*th}
Assume that $X_j,\,j=1,\dots$, are free, $\beta_q<\infty$ with some $q\ge 2$ and 
$m_1=0,\,m_2=1$. Then, for $n\in\mathbb N$,
\begin{equation}\label{2.4a*}
\sup_{x\in\mathbb R}|F_n(x)-w(x)|\le c 
\begin{cases}
\eta_{q1}(n) L_{qn}+n^{-1},&\text{if}\quad \beta_q<\infty,\,2\le q< 3\\
L_{3n},&\text{if}\quad \beta_q<\infty,\,q\ge 3.
\end{cases} 
\end{equation}
\end{theorem} 

In the~case $m_2<\infty$, Theorem~\ref{2.1*th} yields a~type of Free CLT with the~error bound
\begin{equation}\notag
\sup_{x\in\mathbb R}|F_n(x)-w(x)|\le c \Big(\eta_{q1}(n)+n^{-1}\Big),\quad n\in\mathbb N.
\end{equation}
Since $\eta_{q1}(n)\le 10^{3/2}+1,\,n\in\mathbb N$, in the~case $\beta_q<\infty,\,2\le q\le 3$, we obtain from (\ref{2.4a*}) 
the~complete analogue of the Berry-Esseen inequality as well. 
\begin{corollary}\label{2.1*co}
Assume that $X_j,\,j=1,\dots$, are free, $\beta_q<\infty$ with $2<q\le 3$ and 
$m_1=0,\,m_2=1$. Then, for $n\in\mathbb N$,
\begin{equation}\label{2.4a**}
\sup_{x\in\mathbb R}|F_n(x)-w(x)|\le c\, L_{qn}.
\end{equation}
\end{corollary}

In the~case $\beta_3<\infty$ the~inequality~(\ref{2.4a**}) has the~form
\begin{equation}\label{2.4b**}
\sup_{x\in\mathbb R}|F_n(x)-w(x)|\le c \,L_{3n},\qquad n\in\mathbb N.
\end{equation}

The~upper bound (\ref{2.4b**}) sharpens previous results obtained by the~authors~\cite{ChG:2005a}
and V. Kargin~\cite{Ka:2007}.

Theorem~\ref{2.1*th} and Corollary~\ref{2.1*co} are free analogues of Esseen's
inequality in classical probability theory (see~\cite{Pe:1975}, p. 112-120).

\begin{theorem}\label{2.1th}
Assume that $X_j,\,j=1,\dots$, are free, $\beta_q<\infty$ with some $q\ge 3$ and 
$m_1=0,\,m_2=1$. Then, for $n\in\mathbb N$,
\begin{equation}\label{2.4}
\sup_{x\in\mathbb R}|F_n(x)-\mu_{a_n,0,0}((-\infty,x))|\le c 
\begin{cases}
\eta_{q2}(n) L_{qn}
+ L_{3n}^2&\text{if}\quad \beta_q<\infty,\,3\le q<4\\
L_{4n}&\text{if}\quad \beta_q<\infty,\,q\ge 4.
\end{cases}
\end{equation} 
\end{theorem}

\begin{corollary}\label{2.1co}
Under the~assumptions of Theorem~$\ref{2.1th}$
the~following expansion holds
\begin{equation}\label{2.5*}
F_n(x)=w(x)-\frac 13 a_nU_2(x/2)p_w(x)+\rho_{n1}(x),\quad x\in\mathbb R,
\end{equation}
where the~remainder term $\rho_{n1}(x)$ admits the~bound, for $x\in\mathbb R,\,n\in\mathbb N$,
\begin{equation}\label{2.5**}
|\rho_{n1}(x)|\le c 
\begin{cases}
\eta_{q2}(n) L_{qn}
+L_{3n}^2+|a_n|^{3/2}&\text{if}\quad \beta_q<\infty,\,3\le q<4\\
L_{4n}+|a_n|^{3/2}&\text{if}\quad \beta_q<\infty,\,q\ge 4.
\end{cases}
\end{equation} 
\end{corollary}
Note that in the~case $\beta_3<\infty$ the~estimate (\ref{2.5**}) yields the~bound
\begin{equation}\label{2.5***}
|\rho_{n1}(x)|\le c \Big(\eta_{q2}(n)+L_{3n}+|a_n|^{1/2}\Big)L_{3n}, 
\end{equation}
where $\eta_{q2}(n)\to 0$ as $n\to\infty$, and we obtain an~analogue of Edgeworth expansion.

Since $\eta_{q2}(n)\le 101,\,3\le q\le 4,n\in\mathbb N$, the~results (\ref{2.5*}) and (\ref{2.5**})
again yield the~free Berry-Esseen inequality (\ref{2.4b**}) as well.

In addition we obtain from Theorem~\ref{2.1th} the~following bounds.
\begin{corollary}\label{2.1aco}
Under the~assumptions of Theorem~$\ref{2.1th}$ 
\begin{equation}\label{2.5***a}
\sup_{x\in\mathbb R}|F_n(x)-\mu_{a_n,0,0}((-\infty,x))|\le c \,L_{qn}\quad \text{for}\quad n\in\mathbb N
\quad\text{if}\quad \beta_q<\infty,\,3\le q\le 4.
\end{equation}
\end{corollary}

Before formulating the~next result,
denote by $\varsigma_n,\,n>m_4$, a~signed measure with the density
\begin{equation}\label{2.5a***}
p_{\varsigma_n}(x):=(e_n^2(x-a_n)^2-1)p_w(e_n(x-a_n)),\quad x\in\mathbb R,
\end{equation} 
where $e_n:=(1-b_n)/\sqrt{1-d_n}$.
Denote by $\kappa_n,\,n>m_4$, the~{\it signed} measure~ 
$\kappa_n:=\mu_{a_n,b_n,d_n}+\frac 1n\varsigma_n$. 
It is easy to see from results of Section~4 that $\kappa_n$ is a~{\it probability} measure for $n\ge m_4/c$
with some sufficiently small $c$.

\begin{theorem}\label{2.2th}
Assume that $X_j,\,j=1,\dots$, are free random variables, that $\beta_q<\infty$ with some $q\ge 4$
and that $m_1=0,\,m_2=1$. Then, for $n>m_4$, 
\begin{equation}\label{2.6*}
\sup_{x\in\mathbb R}|F_n(x)-\kappa_n((-\infty,x))|\le c 
\begin{cases}
\eta_{q3}(n)L_{qn}+L_{4n}^{3/2}&\text{if}\quad \beta_q<\infty,\,4\le q<5\\
L_{5n}&\text{if}\quad \beta_q<\infty,\,q\ge 5.
\end{cases} 
\end{equation}
\end{theorem}

\begin{corollary}\label{2.2co}
Assume that the~assumptions of Theorem~$\ref{2.2th}$ are satisfied.
Then
\begin{align}\label{2.7*}
F_n(x&+a_n)=w(x)\notag\\
&+\Big(-\frac{a_n^2}2 U_1(x/2)
+\frac{a_n}3(3-U_2(x/2))-\frac{b_n-a_n^2-1/n}{4}U_3(x/2)
\Big)p_w(x)+\rho_{n2}(x),
\end{align}
for all real $x$, where
\begin{equation}\label{2.7a*}
|\rho_{n2}(x)|\le c
\begin{cases}
\eta_{q3}(n)L_{qn}+L_{4n}^{3/2}&\text{if}\quad \beta_q<\infty,\,4\le q<5\\
L_{5n}&\text{if}\quad \beta_q<\infty,\,q\ge 5
\end{cases} 
\quad \text{for}\quad x\in\mathbb R,\quad n\in\mathbb N.
\end{equation}  
If $m_3=0$ this formula has the~following simple form
\begin{equation}\label{2.8}
F_n(x)=w(x)-\frac{m_4-2}{4n}U_3(x/2)p_w(x)+\rho_{n3}(x),
\end{equation}
where $\rho_{n3}(x)$ admits the~bound $(\ref{2.7a*})$.
\end{corollary}

If $m_3\ne 0$, we obtain from (\ref{2.7*}) the following expansion for $F_n(x)$:
\begin{align}
F_n(x)&=w(x)-\frac 13 a_nU_2(x/2)p_w(x)+\Big(\frac{a_n^2}6 U_1(x/2)
-\frac{b_n-a_n^2-1/n}{4}U_3(x/2)\Big)p_w(x)\notag\\
&+ Q_1(x,a_n)+Q_2(x,a_n,b_n,1/n)+\rho_{n4}(x),\quad x\in\mathbb R,\notag 
\end{align}
where 
\begin{align}
&Q_1(x,a_n)=w(x-a_n)-w(x)+a_np_w(x)+\frac{a_n}3(3-U_2(x/2))(p_w(x-a_n)-p_w(x)), \notag\\
&Q_2(x,a_n,b_n,1/n) =\Big(\frac{a_n^2}6 U_1(x/2)-\frac{b_n-a_n^2-1/n}{4}U_3(x/2)\Big)(p_w(x-a_n)-p_w(x))\notag 
\end{align}
and the function $\rho_{n4}(x)$ admits the bound (\ref{2.7a*}). The function $Q_1(x,a_n)$ is 
a function of bounded variation and it is not difficult to verify that
\begin{equation}\label{2.8c}
\frac 1{c}|a_n|^{3/2}\le \sup_{x\in\mathbb R}|Q_1(x,a_n)|\le c|a_n|^{3/2}\quad\text{and}\quad
\frac 1{c}|a_n|^{3/2}\le ||Q_1(x,a_n)||_{TV}\le c|a_n|^{3/2},  
\end{equation}
with some $c\ge 1$. This means that $Q_1(x,a_n)$ is actually of order $n^{-3/4}$.
We shall see that $Q_1(x,a_n)$ can not be cast by Taylor expansion around $x$ into an expansion 
in powers of $n^{-1/2}$ like
(\ref{2.7*}) in terms of $p_w(x)$ and the Chebyshev polynomials which is continuous up to the boundary $\pm 2$
with finite total variation. Moreover $|Q_2(x,a_n,b_n,1/n)|\le cL_{4n}\sqrt{|a_n|},\,x\in\mathbb R$.

Indeed from the formal expansion of $\mu_n$ in (\ref{2.3b}) and (\ref{2.3c}) it follows 
that the~first two summands on the~right-hand
side of (\ref{2.3b}) are Cauchy transforms of the~finite signed measure on
the~right-hand side of (\ref{2.5*}). Moreover, in the~case $m_3=0$
the~first three summands on the~right-hand side of (\ref{2.3b}) are Cauchy 
transforms of the~finite signed measure on
the~right-hand side of (\ref{2.8}). But in the~case $m_3\ne 0$
the~third summand on the~right-hand side of (\ref{2.3b}) can not be a~Cauchy 
transform of a~signed measure $\zeta$ which is finite on every bounded interval and
$\int_{\mathbb R}|\zeta(du)|/(1+|u|)<\infty$. This will be proved in Section~4.  
Therefore, taking into account the formal expansion (\ref{2.3b}), we can not expect 
an expansion of type (\ref{2.7*}) for the function $F_n(x)$ without shift.


\begin{remark}
The methods used in the~proof of Theorems~\ref{2.1th} and \ref{2.2th} still do not yield
a free analogue of Edgeworth asymptotic expansions under the~assumption $\beta_q<\infty, q>5$,
with a~remainder term
of order $O\big(n^{-3/2-\gamma}\big)$ with $\gamma>0$. This problem remains open.
\end{remark}

\begin{remark}
It is known, see for example~\cite{BeVo:1995}, \cite{ChG:2005}, that there is semigroup $\mu_t\in\mathcal M, \,t\ge 1$,
such that $\phi_{\mu_t}(z)=t\phi_{\mu_1}(z)$, where $\phi_{\mu_t}(z)$ are Voiculesku transforms
of the~probability measures $\mu_t$. For the~definition of Voiculesku's transform, see in Section~3.
As before let $m_1=0$ and $m_2=1$. Define a~probability measure $\hat{\mu}_t$ in the~following
way: $\hat{\mu}_t((-\infty,x))=\mu((-\infty,x\sqrt t)),\,x\in\mathbb R$. Theorems~\ref{2.1*th}, 
\ref{2.1th}, \ref{2.2th} and 
their Corollaries
remain valid for $\hat{\mu}_t$ if the~integers $n$ are replaced by $t\ge 1$. One can prove
these results exactly by the~same proof.
\end{remark}

It was proved in~\cite{BelBer:2004} that if the distribution $\mu$ of $X_1$ is not a Dirac measure,
then $F_n(x)$ is Lebesgue absolutely continuous when $n\ge c_1(\mu)$ is sufficiently large. 
Denote by $p_n(x)$ the density of $F_n(x)$. If $\mu$
has a compact support, Voiculescu's result~\cite{Vo:1986} shows that the support of $F_n(x)$
is contained in the interval $[-2-\frac L{\sqrt n},2+\frac L{\sqrt n}]$ for $n\ge 1$, where
$L:=\sup\{|x|:x\in supp(\mu)\}$ (see~\cite{Ka:2007a} as well). 
Our method allows
to obtain an asymptotic expansion for $p_n(x)$ in this case.

\begin{theorem}\label{th7}
Assume that $\mu$ has compact support and $m_1=0,\,m_2=1$. Then, for $n\ge c_1(\mu)$, $p_n(x)$ is a continuous function
such that $p_n(x)\le 2,\,x\in\mathbb R$, and 
\begin{equation}\label{asden1}
p_n(x+a_n)=\Big(1+\frac 12 d_n-a_n^2-\frac 1{n}-a_nx-\Big(b_n-a_n^2-\frac 1{n}\Big)x^2\Big)
p_w(e_nx)+\frac {c(\mu)\theta}{n^{3/2}\sqrt{4-(e_nx)^2}}
\end{equation}
for $x\in[-\frac 2{e_n}+h,\frac 2{e_n}-h]$, where 
$h=\frac {c_2(\mu)}{n^{3/2}}$. Moreover,
as a simple consequence of $(\ref{asden1})$, the following inequality holds
\begin{equation}\label{asden2}
\int_{\mathbb R\setminus[-\frac 2{e_n}+h,\frac 2{e_n}-h]}p_n(x+a_n)\,dx\le \frac {c(\mu)}{n^{3/2}}. 
\end{equation}
\end{theorem}
Here and in the sequel we denote by $\theta$ a real-valued quantity such that $|\theta|\le 1$.

We see that the remainder term in (\ref{asden1}) has the order $\frac 1{n\sqrt{M_n}}$ on the interval 
$[-\frac 2{e_n}+\frac{M_n}n,\frac 2{e_n}-\frac{M_n}n]$
with any $M_n$ such that $M_n\to\infty$ and $\frac{M_n}n\to 0$ as $n\to\infty$. 

Now we shall state some consequences of Theorem~\ref{th7}.
\begin{corollary}\label{corth7.1}
Let $\mu_n$ be the distribution of $Y_n$ from $(\ref{2.1})$ with bounded identically distributed  
free summands $X_1,\dots,X_n$ such that $m_1=0$ and $m_2=1$.
If $m_3\ne 0$, then
\begin{equation}\label{2.9}
\int_{\mathbb R}|p_n(x)-p_w(x)|\,dx=\frac {2|m_3|}{\pi\sqrt n}+\theta\big(c|a_n|^{3/2}+\frac {c(\mu)}n\big),
\quad n\ge c_1(\mu). 
\end{equation}
If $m_3=0$, then
\begin{equation}\label{2.10}
\int_{\mathbb R}|p_n(x)-p_w(x)|\,dx=\frac {2|m_4-2|}{\pi n}+\theta\frac{c(\mu)}{n^{3/2}},\quad n\ge c_1(\mu). 
\end{equation}
\end{corollary}


Recall that, if the random variable $X$ has density $f$, then the classical entropy of a
distribution of $X$ is defined as $h(X)=-\int_{\mathbb R}f(x)\log f(x)\,dx$,
provided the positive part of the integral is finite. Thus we have $h(X)\in[-\infty,\infty)$.

A much stronger statement than the classical CLT -- the entropic central limit theorem --
indicates that, if for some $n_0$, or equivalently, for all $n\ge n_0$, $Y_n$ from (\ref{2.1})
have absolutely continuous distributions with finite entropies $h(Y_n)$, then there is convergence
of the entropies, $h(Y_n)\to h(Y)$, as $n\to \infty$, where $Y$ is a standard Gaussian random variable. 
This theorem is due to Barron~\cite{Ba:1986}. Recently Bobkov, Chistyakov and G\"otze~\cite{BChG:2011}
found the rate of convergence in the classical entropic CLT.


Let $\nu$ be a probability measure on $\mathbb R$. The quantity
\begin{equation}\notag
\chi(\nu)=\int\int_{\mathbb R\times \mathbb R}\log|x-y|\,\nu(dx)\nu(dy)+\frac 34+\frac 12\log 2\pi,
\end{equation}
called free entropy of $\nu$, was introduced by Voiculescu in~\cite{Vo:1993}. Free entropy $\chi$ behaves like 
the classical entropy $h$. 
In particular, the free entropy is maximized 
by the standard semicircle measure $\mu_w$ with the value $\chi(\mu_w)=\frac 12\log 2\pi e$ among all 
probability measures with variance one, see~\cite{HiPe:2000}, \cite{Vo:1997}. Wang~\cite{Wa:2010}
has proved the free analogue of Barron's result. We give the rate of convergence in the free CLT
for bounded free random variables.
\begin{corollary}\label{corth7.2}
Let $\mu_n$ be the distribution of $Y_n$ from $(\ref{2.1})$ with bounded identically distributed  
free summands $X_1,\dots,X_n$ such that $m_1=0$ and $m_2=1$. For $n\ge c_1(\mu)$,
\begin{equation}\notag
\chi(\mu_n)=\int\int_{\mathbb R\times \mathbb R}\log|x-y|\,p_n(x)p_n(y)\,dxdy+\frac 34+\frac 12\log 2\pi
=\chi(\mu_w)-\frac {m_3^2}{6n}+\theta\frac{c(\mu)}{n^{3/2}}.  
\end{equation} 
\end{corollary}

Suppose that a measure $\nu$ has a density $p$ in $L^3(\mathbb R)$. Then, following Voiculescu~\cite{Vo:1993},
the free Fisher information of $\nu$ is
\begin{equation}\notag
 \Phi(\nu)=\frac{4\pi^2}3\int_{\mathbb R}p(x)^3\,dx.
\end{equation}
It is well-known that $\Phi(w)=1$. Moreover, the free Cram\'er-Rao inequality shown in \cite{Vo:1993}
says that $\Phi(\nu)\int_{\mathbb R}\Big(x-\int_{\mathbb R}\,u\,\nu(du)\Big)^2\nu(dx)\ge 1$, and equality
holds if and only if $\nu$ is a measure with a semicircle distribution function.
We obtain the following result for bounded free random variables.
\begin{corollary}\label{corth7.3}
Let $\mu_n$ be the distribution of $Y_n$ from $(\ref{2.1})$ with bounded identically distributed  
free summands $X_1,\dots,X_n$ such that $m_1=0$ and $m_2=1$. For $n\ge c_1(\mu)$,
\begin{equation}\label{2.11}
\Phi(\mu_n)=\frac{4\pi^2}3\int_{\mathbb R}p_n(x)^3\,dx=\Phi(\mu_w)+\frac{m_3^2}n+\theta\frac{c(\mu)}{n^{3/2}}.  
\end{equation} 
\end{corollary}

\section{Auxiliary results}

We need results about some classes of analytic functions
(see {\cite{Akh:1965}, Section~3, and {\cite{AkhG:1963}},
Section~6, \S 59). 

The~class $\mathcal N$ (Nevanlinna, R.) is the~class of analytic 
functions $f(z):\mathbb C^+\to\{z: \,\Im z\ge 0\}$.
For such functions there is the~integral representation
\begin{equation}\label{3.1}
f(z)=a+bz+\int_{\mathbb R}\frac{1+uz}{u-z}\,\tau(du)=
a+bz+\int_{\mathbb R}\Big(\frac 1{u-z}-\frac u{1+u^2}\Big)(1+u^2)
\,\tau(du),\quad z\in\mathbb C^+,
\end{equation}
where $b\ge 0$, $a\in\mathbb R$, and $\tau$ is a~nonnegative finite
measure. Moreover, $a=\Re f(i)$ and $\tau(\mathbb R)=\Im f(i)-b$.   
From this formula it follows that 
\begin{equation}\label{3.2}
f(z)=(b+o(1))z
\end{equation} 
for $z\in\mathbb C^+$
such that $|\Re z|/\Im z$ stays bounded as $|z|$ tends to infinity (in other words
$z\to\infty$ nontangentially to $\mathbb R$).
Hence if $b\ne 0$, then $f$ has a~right inverse $f^{(-1)}$ defined
on the~region 
$$
\Gamma_{\alpha,\beta}:=\{z\in\mathbb C^+:|\Re z|<\alpha \Im z,\,\Im z>\beta\}
$$
for any $\alpha>0$ and some positive $\beta=\beta(f,\alpha)$.

A~function $f\in\mathcal N$ admits the~representation
\begin{equation}\label{3.3}
f(z)=\int_{\mathbb R}\frac{\sigma(du)}{u-z},\quad z\in\mathbb C^+,
\end{equation}
where $\sigma$ is a~finite nonnegative measure, if and only if
$\sup_{y\ge 1}|yf(iy)|<\infty$.                                      

For $\mu\in\mathcal M$, consider its Cauchy transform $G_{\mu}(z)$
(see (\ref{2.3a})).
The measure $\mu$ can be recovered from $G_{\mu}(z)$ as
the weak limit of the measures
\begin{equation}\notag
\mu_y(dx)=-\frac 1{\pi}\Im G_{\mu}(x+iy)\,dx,\quad x\in\mathbb R,\,\,y>0,
\end{equation}
as $y\downarrow 0$. If the function $\Im G_{\mu}(z)$ is continuous at $x\in\mathbb R$,
then the probability distribution function $D_{\mu}(t)=\mu((-\infty,t))$ is differentiable
at $x$ and its derivative is given by 
\begin{equation}\label{3.4}
D_{\mu}'(x)=-\Im G_{\mu}(x)/\pi. 
\end{equation}
This inversion formula
allows to extract the density function of the measure $\mu$ from its Cauchy transform.

Following Maassen~\cite{Ma:1992} and Bercovici and 
Voiculescu~\cite{BeVo:1993}, we shall consider in the~following
the~ {\it reciprocal Cauchy transform}
\begin{equation}\label{3.5}
F_{\mu}(z)=\frac 1{G_{\mu}(z)}.
\end{equation}
The~corresponding class of reciprocal Cauchy
transforms of all $\mu\in\mathcal M$ will be denoted by $\mathcal F$.
This class coincides with the~subclass of Nevanlinna functions $f$
for which $f(z)/z\to 1$ as $z\to\infty$ nontangentially to $\mathbb R$.
Indeed, reciprocal Cauchy transforms of probability measures have obviously such 
property. Let $f\in\Cal N$ and $f(z)/z\to 1$ as $z\to\infty$ nontangentially 
to $\mathbb R$. Then, by (\ref{3.2}), $f$ admits the~representation~(\ref{3.1}) with $b=1$.
By (\ref{3.2}) and (\ref{3.3}), $-1/f(z)$ admits 
the~representation~(\ref{3.3}) with $\sigma\in\Cal M$.

The~functions $f$ of the~class $\mathcal F$ satisfy the~inequality
\begin{equation}\label{3.5**}
\Im f(z)\ge \Im z,\qquad z\in\mathbb C^+.
\end{equation}

The~function $\phi_{\mu}(z)=F_{\mu}^{(-1)}(z)-z$ is called
the~Voiculescu transform of $\mu$ and
$\phi_{\mu}(z)$ is an~analytic function on $\Gamma_{\alpha,\beta}$ 
with the~property 
$\Im \phi_{\mu}(z)\le 0$ for $z\in \Gamma_{\alpha,\beta}$, where 
$\phi_{\mu}(z)$ is defined. 
On the~domain $\Gamma_{\alpha,\beta}$, where the~functions $\phi_{\mu_1}(z)$,
$\phi_{\mu_2}(z)$, and $\phi_{\mu_1\boxplus\mu_2}(z)$ are defined, we have
\begin{equation}\label{3.6}
\phi_{\mu_1\boxplus\mu_2}(z)=\phi_{\mu_1}(z)+\phi_{\mu_2}(z).
\end{equation} 
This relation  for the~distribution $\mu_1\boxplus\mu_2$ of $X+Y$,
where $X$ and $Y$ are free random variables, is due to
Voiculescu~\cite{Vo:1986}
for the case of  compactly supported measures.
The~result was extended by Maassen~\cite{Ma:1992} to measures 
with finite variance; the~general case was proved by 
Bercovici and Voiculescu~\cite{BeVo:1993}.

Assume that $\beta_k<\infty$ for some $k\in\mathbb N$. Then
$$
G_{\mu}(z)=\frac 1z+\frac{m_1}{z^2}+\dots+\frac{m_k}{z^{k+1}}
+o\Big(\frac 1{z^{k+1}}\Big),\quad z\to\infty,\,\,z\in\Gamma_{\alpha,1}.
$$
It follows from this relation (see for example \cite{Ka:2007a}) that
\begin{equation}\label{3.7}
\phi_{\mu}(z)=\alpha_1+\frac{\alpha_2}{z}+\dots
+\frac{\alpha_{k}}{z^{k-1}}+o\Big(\frac 1{z^{k-1}}\Big),
\quad z\to\infty,\,\,z\in\Gamma_{\alpha,1}.
\end{equation}
We call the~coefficients $\alpha_m,\,m=1,\dots,k$, the free cumulants
of the~probability measure $\mu$. It is easy to see that $\alpha_1=m_1,
\alpha_2=m_2-m_1^2,\,\alpha_3(\mu)=m_3-3m_1m_2
+2m_1^3$. In the~case $m_1=0$ and $m_2=1$ we have $\alpha_1=0,\,
\alpha_2=1,\,\alpha_3=m_3$ and $\alpha_4=m_4-2$.

If $\mu\in\mathcal M$ has moments of any order, 
that is $\beta_k<\infty$ for any $k\in\mathbb N$, then there exist cumulants
$\alpha_m,\,m=1,\dots$, and we can consider the~formal power series
\begin{equation}\label{3.8}  
\phi_{\mu}(z)=\sum_{m=1}^{\infty}\frac{\alpha_{m}}{z^{m-1}}.
\end{equation}
In addition $\phi_{\mu}(z)$ satisfies (\ref{3.7}) for any fixed $k\in\mathbb N$.
If $\mu$ has a~bounded support, $\phi_{\mu}(z)$ is an~analytic function
on the~domain $|z|>R$ with some $R>0$ and the~series (\ref{3.8}) converges
absolutely and uniformly for such $z$.

Voiculescu~\cite{Vo:1993} showed for compactly supported probability measures that
there exist unique functions $Z_1, Z_2\in\mathcal F$ such that
$G_{\mu_1\boxplus\mu_2}(z)=G_{\mu_1}(Z_1(z))=G_{\mu_2}(Z_2(z))$ 
for all $z\in\mathbb C^+$.
Using Speicher's combinatorial approach~\cite{Sp:1998} to freeness,
Biane~\cite{Bi:1998} proved this result in the~general case.

Chistyakov and G\"otze \cite {ChG:2005}, Bercovici and Belinschi~\cite{BelBer:2007},
Belinschi~\cite{Bel:2008}, 
proved, using complex analytic methods, that
there exist unique functions $Z_1(z)$ and $Z_2(z)$ in the~class 
$\mathcal F$ such that, for $z\in\mathbb C^+$, 
\begin{equation}\label{3.9}
z=Z_1(z)+Z_2(z)-F_{\mu_1}(Z_1(z))\quad\text{and}\quad
F_{\mu_1}(Z_1(z))=F_{\mu_2}(Z_2(z)). 
\end{equation}
The~function $F_{\mu_1}(Z_1(z))$ belongs again to the~class $\mathcal F$ and 
there exists 
$\mu\in\mathcal M$ such that
$F_{\mu_1}(Z_1(z)) =F_{\mu}(z)$, where $F_{\mu}(z)=1/G_{\mu}(z)$ and 
$G_{\mu}(z)$ is the~Cauchy transform as in (\ref{2.3a}). 
We can define the~additive free convolution in the~following way
$\mu_1\boxplus\mu_2:=\mu$. 
The~measure $\mu$ depends on $\mu_1$ and $\mu_2$ only.
The~relation (\ref{3.6}) follows immediately from (\ref{3.9}) and we see that
this definition coincides with the~Voiculescu, Bercovici, Maassen definition.
Hence we have the~equivalence of a~''characteristic function'' approach
and a~probabilistic approach to the~definition of 
the~additive free convolution.

Specializing to $\mu_1=\mu_2=\dots=\mu_n=\mu$ write $\mu_1\boxplus\dots\boxplus\mu_n=
\mu^{n\boxplus}$.
The~relation (\ref{3.9}) admits the~following
consequence (see for example \cite{ChG:2005}).

\begin{proposition}\label{3.3pro}
Let $\mu\in\mathcal M$. There exists a~unique function $Z\in\mathcal F$ 
such that
\begin{equation}\label{3.10}
z=nZ(z)-(n-1)F_{\mu}(Z(z)),\quad z\in\mathbb C^+,
\end{equation}
and $F_{\mu^{n\boxplus}}(z)=F_{\mu}(Z(z))$.
\end{proposition}

The~next lemma was proved in \cite{ChG:2005}.
\begin{lemma}\label{l7.3}
Let $g:\mathbb C^+\to \mathbb C^-$ be analytic with
\begin{equation}\label{7.7}
\liminf_{y\to+\infty}\frac {|g(iy)|}y=0.
\end{equation}
Then the~function $f:\mathbb C^+\to\mathbb C$ defined via $z\mapsto z+g(z)$
takes every value in $\mathbb C^+$ precisely once. The~inverse
$f^{(-1)}:\mathbb C^+\to \mathbb C^+$ thus defined is in the~class $\mathcal F$.
\end{lemma}

This lemma generalizes a~result of Maassen~\cite{Ma:1992} (see Lemma~2.3).
Maassen proved Lemma~\ref{l7.3} under the~additional restriction $|g(z)|
\le c(g)/\Im z$ for $z\in\mathbb C^+$, where $c(g)$ is a~constant depending
on $g$.

Using the~representation (\ref{3.1}) for $F_{\mu}(z)$ we obtain
\begin{equation}\label{7.8}
F_{\mu}(z)=z+\Re F_{\mu}(i)+\int_{\mathbb R}\frac {(1+uz)\,\tau(du)}{u-z},
\quad z\in\mathbb C^+,
\end{equation}
where $\tau$ is a~nonnegative measure such that $\tau(\mathbb R)=\Im F_{\mu}(i)-1$.
Denote $z=x+iy$, where $x,y\in\mathbb R$. We see that, for $\Im z>0$, 
\begin{equation}\notag
\Im \Big(nz-(n-1)F_{\mu}(z)\Big)=y\Big(1-(n-1)I_{\mu}(x,y)\Big),\quad\text{where}\quad
I_{\mu}(x,y):=\int_{\mathbb R}\frac{(1+u^2)\,\tau(du)}{(u-x)^2+y^2}.
\end{equation}
For every real fixed $x$, consider the~equation
\begin{equation}\label{7.9}
 y\Big(1-(n-1)I_{\mu}(x,y)\Big)=0,\quad y>0.
\end{equation}
Since $y\mapsto I_{\mu}(x,y),\,y>0$, is positive and monotone, and decreases to $0$ as $y\to\infty$,
it is clear that the~equation (\ref{7.9}) has at most one positive solution. If such a~solution exists, denote it
by $y_n(x)$.
Note that (\ref{7.9}) does not have a~solution $y>0$ for any given $x\in\mathbb R$ if and only if
$I_{\mu}(x,0)\le 1/(n-1)$.
Consider the~set $S:=\{x\in\mathbb R:I_{\mu}(x,0)\le 1/(n-1)\}$. We put $y_n(x)=0$ for $x\in S$. 
By Fatou's lemma,
$I_{\mu}(x_0,0)\le\lim\inf_{x\to x_0}I_{\mu}(x,0)$ for any given $x_0\in\mathbb R$, hence
the~set $S$ is closed.
Therefore $\mathbb R\setminus S$ is the~union of finitely or countably many intervals $(x_k,x_{k+1}),\,x_k<x_{k+1}$.
The~function $y_n(x)$ is continuous on the~interval $(x_k,x_{k+1})$. 
Since the~set $\{z\in\mathbb C^+:n\Im z-(n-1)\Im F_{\mu}(z)>0\}$
is open, we see that $y_n(x)\to 0$ if $x\downarrow x_k$ and $x\uparrow x_{k+1}$.
Hence the~curve $\gamma_n$ given by the~equation $z=x+iy_n(x),\,x\in\mathbb R$, is continuous and simple.

Consider the~open domain $D_n:=\{z=x+iy,\,x,y\in\mathbb R: y>y_n(x)\}$.
\begin{lemma}\label{l7.4}
Let $Z\in\mathcal F$ be the~solution of the~equation $(\ref{3.10})$.
The function $Z(z)$ maps $\mathbb C^+$
conformally onto $D_n$. 
Moreover the~function $Z(z),\,z\in\mathbb C^+$, 
is continuous up to the~real axis and it establishes a homeomorphism between the real axis and 
the~curve $\gamma_n$.   
\end{lemma}
\begin{proof}
We obtain from (\ref{3.10}) the~formula
\begin{equation}\label{7.10}
Z^{(-1)}(z)=nz-(n-1)F_{\mu}(z) 
\end{equation}
for $z\in\Gamma_{\alpha,\beta}$ with some $\alpha,\beta>0$. By this formula we may continue
the~function $Z^{(-1)}(z)$ as an~analytic function to $\mathbb C^+$. Using the~representation (\ref{7.8})
for the~function $F_{\mu}(z)$,
we note that $Z^{(-1)}(z)=z+g(z),\,z\in\mathbb C^+$, where $g(z)$ is analytic on $\mathbb C^+$ and satisfies
the~assumptions of Lemma~\ref{l7.3}. 
By Lemma~\ref{l7.3}, we conclude that the~function $Z^{(-1)}(z)$ takes every value in $\mathbb C^+$
precisely once. Moreover, as it is easy to see, $Z^{(-1)}(D_n)=\mathbb C^+$.
The~inverse $Z(z)$ gives us the~conformal mapping
of $\mathbb C^+$ onto $D_n$. By the~well-known results of the~theory of analytic functions (see~\cite{Mar:1965}), 
$Z(z)$ is continuous up to the~real axis and it 
establishes a homeomorphism between the real axis and 
the~curve $\gamma_n$.   
\end{proof}
\begin{lemma}\label{l7.5}
Let $\mu$ be a~probability measure such that $m_1=0,m_2=1$. Assume that $\rho_2(\mu,\sqrt{(n-1)/8})\le 1/10$ for some
positive integer $n\ge 10^3$.
Then the~following inequality holds
\begin{equation}\label{7.11*}
|Z(z)|\ge \sqrt{(n-1)/8},\qquad z\in\mathbb C^+, 
\end{equation} 
where $Z\in\mathcal F$ is the~solution of the~equation $(\ref{3.10})$.
\end{lemma}
\begin{proof}
Write 
\begin{equation}\label{7.11}
G_{\mu}(z)=\frac 1z+\frac{r(z)}{z^2},\quad z\in\mathbb C^+\quad
\text{where}\quad r(z):=\int_{\mathbb R}\frac{u^2\,\mu(du)}{z-u}.
\end{equation}
It is obvious that $|r(z)|\le 1/\Im z,\,z\in\mathbb C^+$. Rewrite (\ref{7.9}) in the~form
\begin{equation}\notag
\Im z\Big(1+(n-1)\Big(1-\frac 1{\Im z}\Im\frac 1{G_{\mu}(z)}\Big)\Big)=0.
\end{equation}
Let us show that 
\begin{equation}\label{7.12}
y_n(x)> \sqrt{(n-1)/8}\quad\text{for}\quad |x|\le\sqrt{(n-1)/8}.
\end{equation}

In order to prove this inequality we shall establish that 
\begin{equation}\label{7.13}
(n-1)\Big(1-\frac 1{\Im z}\Im\frac 1{G_{\mu}(z)}\Big)<-1
\end{equation} 
for $|z|=\frac 12\sqrt{(n-1)}$ and 
$|\Re z|\le\sqrt{(n-1)/8}$. Indeed, since the~function 
\begin{equation}\notag
-I_{\mu}(\Re z,\Im z)=(n-1)\Big(1-\frac 1{\Im z}\Im\frac 1{G_{\mu}(z)}\Big),\quad \Im z>0,
\end{equation}
is negative and monotone, and increases to $0$
as $\Im z\to\infty$, (\ref{7.12}) follows from (\ref{7.13}).

We have, for the~$z$ considered above,
\begin{equation}\notag
\frac 1{G_{\mu}(z)}=\Big(\frac 1{z}+\frac{r(z)}{z^2}\Big)^{-1}=z-r(z)+r_1(z),
\end{equation}
where $r_1(z)$ admits the~upper bound $|r_1(z)|\le 2(|z|(\Im z)^2)^{-1}\le 32/(n-1)^{3/2}$.

Using the~previous formula, we easily obtain the~relation, for the~same $z$,
\begin{equation}\label{7.14}
-I_{\mu}(\Re z,\Im z)=(n-1)\frac{\Im r(z)}{\Im z}+r_2(z)=
-(n-1)\int_{\mathbb R}\frac{u^2\,\mu(du)}{(u-\Re z)^2+(\Im z)^2}+r_2(z),
\end{equation}
where $r_2(z)$ admits the~upper bound $|r_2(z)|\le 32\sqrt 8/(n-1)<1/6$.
Hence we have, for the~same $z$,
\begin{align}
-I_{\mu}(\Re z,\Im z)&\le -(n-1)\int_{[-\Im z,\Im z]}\frac{u^2\,\mu(du)}{(u-\Re z)^2+(\Im z)^2}+r_2(z)\notag\\
&\le-\frac{n-1}{3|z|^2}(1-\rho_2(\mu,\Im z))
+r_2(z)\le -\frac{n-1}{3|z|^2}(1-\rho_2(\mu,\sqrt{(n-1)/8}))+r_2(z)\notag\\
&<-\frac 65+\frac 16<-1\notag
\end{align}
and (\ref{7.13}) is proved.

The~assertion of the~lemma follows immediately from (\ref{7.12}).
\end{proof}

Denote by $\Delta(\kappa',\kappa'')$ the~Kolmogorov
distance between the~finite signed measures $\kappa'$ and $\kappa''$ 
such that $\kappa'((-\infty,x))\to 0$ and $\kappa''((-\infty,x))\to 0$ as $x\to-\infty$, i.e.,
$$
\Delta(\kappa',\kappa''):=\sup_{x\in\mathbb R}|\kappa'((-\infty,x))-\kappa''((-\infty,x))|.
$$

We need the~following result of Bercovici-Voiculescu~\cite{BeVo:1993}.
\begin{proposition}\label{3.3b}
If $\mu,\mu',\nu$ and $\nu'$ are probability measures, then
\begin{equation}\notag
\Delta(\mu\boxplus\nu,\mu'\boxplus\nu')\le\Delta(\mu,\mu')+\Delta(\nu,\nu').
\end{equation}
\end{proposition}

In addition the~following proposition holds (see~\cite{Pe:1975}, p.139).
\begin{proposition}\label{3.3c}
If $3\le m\le k$, then the~Lyapunov fractions $L_{mn}$ and $L_{kn}$ satisfy the inequality: 
$L_{mn}^{1/(m-2)}\le L_{kn}^{1/(k-2)}$. 
\end{proposition}

\section{Properties of free Meixner measures}

Saitoh and Yoshida~\cite{SaYo:2001} have proved that
the~absolutely continuous part of the~free Meixner measure $\mu_{a,b,d},a\in\mathbb R,b<1,d<1$, is 
given by
\begin{equation}\label{2.3g}
\frac{\sqrt{4(1-d)-(1-b)^2(x-a)^2}}{2\pi f(x)},
\end{equation}
when $a-2\sqrt{1-d}/(1-b)\le x\le a+2\sqrt{1-d}/(1-b)$, where 
\begin{equation}\notag
f(x):=bx^2+a(1-b)x+1-d;  
\end{equation} 
the~measure may have a~discrete part $\mu_D$ in the~following cases:

1. if $f(x)$ has two real roots $y_1\ne y_2$, then
\begin{equation}\label{2.3*}
\mu_D:=\lambda_1\delta_{y_1}+\lambda_2\delta_{y_2},
\end{equation}
where
\begin{equation}\label{2.3**}
\lambda_j:=\frac{1}{\sqrt{a^2(1-b)^2-4b(1-d)}}\Big(\frac{1-d}{|y_j|}
-|y_j|\Big)_+,\quad j=1,2,
\end{equation}

2. if $b=0$ and $a\ne 0$, then
\begin{equation}\label{2.3***}
\mu_D:=\Big(1-\frac{1-d}{a^2}\Big)_+\delta_y,\quad\text{where}\quad
y:=-\frac{1-d}{a}.
\end{equation}

Recall that $\delta_y$ with $y\in\mathbb R$ is a Dirac measure
concentrated at the point $y$.

Saitoh and Yoshida proved as well that for $0\le b<1$ the (centered) 
free Meixner measure
$\mu_{a,b,d}$ is $\boxplus$-infinitely divisible.
Note (see Bo\.zejko and Bryc~\cite{BoBr:2006}) that $\mu_{a,b,d}=\mu_w$ 
if $a=b=d=0$; $\mu_{a,b,d}$
is the~free Poisson type measure, which is also known as Marchenko-Pastur 
measure~\cite{MaPa:1967}, if $b=d=0$ and $a\ne 0$, and $\mu_{a,0,d}$ with $a\ne 0, d\ne 0$ 
is the~shifted free Poisson type measure ; $\mu_{a,b,d}$ is
the~free Pascal (negative binomial) type measure if  $b>0$ and $a^2(1-b)^2>4b(1-d)$;
$\mu_{a,b,d}$ is the~free gamma type measure if $b>0$ and $a^2(1-b)^2=4b(1-d)$;
$\mu_{a,b,d}$ is the~pure free Meixner type measure if $b>0$ and $a^2(1-b)^2<4b(1-d)$.

Now assume that $m_4<\infty, m_1=0,m_2=1$ and $n\ge 3m_4$. By the~well-known moment inequality 
\begin{equation}\notag
\begin{vmatrix}
1&m_1&m_2\\
m_1&m_2&m_3\\
m_2&m_3&m_4
\end{vmatrix}
\ge 0
\end{equation}
(see~\cite{Akh:1965}), we conclude that $m_4-1-m_3^2\ge 0$. Therefore the~lower bounds
$b_n\ge 0$ and $d_n>0$ hold. 
In addition
we have $|a_n|\le 1/\sqrt 3,b_n\le 1/3$ and $d_n\le 1/3$.
Consider the~measures $\mu_{a_n,b_n,d_n}$.  
These~measures 
may be the~free Pascal, the free gamma and the pure free Meixner type measures.

Let $b_n>0$. Note that the~polynomial $f_n(x)=b_n x^2+a_n(1-b_n)x+1-d_n$ has two
real roots $y_{1n}$ and $y_{2n}$ in the~case $a_n^2(1-b_n)^2-4b_n(1-d_n)>0$ and these
roots have the~same sign. By the~relation 
\begin{equation}\notag
\frac 1{|y_1|}+\frac 1{|y_2|}=\frac{|a_n|(1-b_n)}{1-d_n}\le\frac {\sqrt 3}2<1,
\end{equation}
one can deduce the~inequalities $|y_{jn}|\ge 1,\,j=1,2$. 
Using (\ref{2.3*}) and (\ref{2.3**}) we see that the~discrete part 
of $\mu_{a_n,b_n,d_n}$ is equal to zero.
Let $b_n=0$ and $a_n\ne 0$. We see from (\ref{2.3***}) that in this case 
the~discrete part of $\mu_{a_n,b_n,d_n}$ is equal to zero as well.

Thus, in the~considered case it follows from Saitoh and Yoshida's results that the~probability measure
$\mu_{a_n,b_n,d_n}$ is $\boxplus$-infinitely divisible 
and it is absolutely continuous with a~density of the~form (\ref{2.3g}) where
$a=a_n, b=b_n,d=d_n$.

Assume that $\beta_3<\infty,m_1=0,m_2=1$ and $n\ge m_3^2$, i.e., $|a_n|\le 1$. In this case the~probability measure
$\mu_{a_n,0,0}$ is absolute continuous with a~density of the~form (\ref{2.3g}) where $a=a_n,b=0,d=0$.
In addition, by Saitoh and Yoshida's results, $\mu_{a_n,0,0}$ is $\boxplus$-infinitely divisible.

\section{Formal Asymptotic expansion 
in the Free CLT}

In this section we deduce formula (\ref{2.3b}).

By Proposition~\ref{3.3pro}, there exists $Z(z)\in\mathcal F$
such that (\ref{3.10}) holds, and $F_{\mu^{n\boxplus}}(z)=F_{\mu}(Z(z))$.
Hence $F_{\mu_n}(z)=F_{\mu}(\sqrt n S_n(z))/\sqrt n,\,z\in\mathbb C^+$, 
where $S_n(z):=Z(\sqrt n z)/\sqrt n$. Using the~Voiculescu transform $\phi_{\mu}(z)$
(see (\ref{3.6}), this relation implies
that 
$$
S_n(z)=F_{\mu_n}(z)+\phi_{\mu}(\sqrt n F_{\mu_n}(z))/\sqrt n
$$
for $z\in\Gamma_{\alpha,\beta}$ with some $\alpha,\beta>0$.
On the~other hand we conclude from (\ref{3.10}) that
$$
S_n(z)=\frac zn+\frac {n-1}n F_{\mu_n}(z),\qquad z\in\mathbb C^+.
$$
The~last two equations give us
\begin{equation}\label{3.1**}
F_{\mu_n}(z)+\sqrt n\phi_{\mu}(\sqrt n F_{\mu_n}(z))=z,\qquad z
\in\Gamma_{\alpha,\beta}.
\end{equation}
Consider the~function $f(z):=z+\sqrt n\phi_{\mu}(\sqrt n z),\,z
\in\Gamma_{\alpha,\beta'}$ with some $\beta'\ge \beta$, 
and define the~function 
\begin{equation}\label{3.1*}
g(z):=\frac 12\Big(f(z)+\sqrt{f^2(z)-4}\Big),\quad z\in\Gamma_{\alpha,\beta'}, 
\end{equation}
where we choose the~branch of the~square root by the~condition
$\Im g(z)>0$ for $z\in\Gamma_{\alpha,\beta'}$.
It is easy to see that $g(z)=z(1+o(1))$ as $z\to\infty$ nontangentially
to $\mathbb R$. In addition, by (\ref{3.1**}), $g(z)$ satisfies the~relation
\begin{equation}\label{3.2*}
g(F_{\mu_n}(z))+\frac 1{g(F_{\mu_n}(z))}=f(F_{\mu_n}(z))=z,
\qquad z\in\Gamma_{\alpha,\beta'}. 
\end{equation}
We deduce from (\ref{3.2*}) that 
$$
g(F_{\mu_n}(z))=\frac 12\Big(z+\sqrt{z^2-4}\Big)=F_{\mu_w}(z),\qquad 
z\in\Gamma_{\alpha,\beta'}.
$$
Recall that we denote by $\mu_w$ the~semicircle measure.

Since the~function $g(z)$ has a~right inverse $g^{(-1)}(z)$ in
$\Gamma_{\alpha,\beta''}$ with some $\beta''\ge \beta'$, we have
\begin{equation}\label{3.3*}
F_{\mu_n}(z)=g^{(-1)}(F_{\mu_w}(z)),\qquad 
z\in\Gamma_{\alpha,\beta'''},\quad\text{where}\quad \beta'''\ge \beta''.
\end{equation}

Let $\mu\in\mathcal M$ such that all moments of $\mu$ exist. In addition let
$m_1=0$ and $m_2=1$.
Consider the~formal power series in $z$
\begin{equation}\label{3.4*}
\sqrt n\phi_{\mu}(\sqrt n z):=\sum_{k=1}^{\infty}
\frac{\alpha_{k+1}}{n^{(k-1)/2}z^k}, 
\end{equation}
where $\alpha_k,\,k=1,2,\dots$, are free cumulants of the~measure $\mu$
and the~formal power series of $g$:
\begin{equation}\label{3.4**}
g(z)=z+\sum_{k=0}^{\infty}\frac{a_k}{z^k}.
\end{equation} 
In our case $\alpha_1=0,\alpha_2=1,\alpha_3=m_3$ and $\alpha_4=
m_4-2$. Using
(\ref{3.2*}) and (\ref{3.4*}), (\ref{3.4**}) we obtain the~following 
relation for the~considered formal power series
$$
z+\sum_{k=0}^{\infty}\frac{a_k}{z^k}+\frac 1z\Big(1-
\sum_{k=0}^{\infty}\frac{a_k}{z^{k+1}}+\Big(\sum_{k=0}^{\infty}\frac{a_k}
{z^{k+1}}\Big)^2
-\dots\Big)=z+\sum_{k=1}^{\infty}\frac{\alpha_{k+1}}{n^{(k-1)/2}z^k}.
$$
It follows from this relation that $a_0=0$, $a_1=0$ and 
\begin{equation}\label{3.4*a}
a_k-a_{k-2}+\sum_{s=0}^{k-3}a_sa_{k-s-3}-\dots+(-1)^{k-1}a_0^{k-1}
=\frac{\alpha_{k+1}}{n^{(k-1)/2}},\quad k=2,3,\dots.
\end{equation}
We have from (\ref{3.4*a}) the~relations $a_2=\alpha_3/\sqrt n,\,a_3=\alpha_4/n$.
In addition we obtain from (\ref{3.4*a}) by induction that
\begin{equation}\label{3.4*b}
a_{2s}=\frac{\alpha_3}{\sqrt n}+O\Big(\frac 1{n^{3/2}}\Big),\quad\text{and}\quad
a_{2s+1}=\frac{\alpha_4}n-\frac {(s-1)(s-2)}2\frac{\alpha_3^2}{n}
+O\Big(\frac 1{n^{3/2}}\Big)
\end{equation}
as $n\to\infty$ for $s=2,\dots$.

Now consider the~formal power series for the~right inverse $g^{(-1)}(z)$
\begin{equation}\notag
g^{(-1)}(z)=z+\sum_{k=0}^{\infty}\frac{b_k}{z^k}.
\end{equation}
Rewrite the~relation $g(g^{(-1)}(z))=z$ in the~form
\begin{equation}\notag
z+\sum_{m=0}^{\infty}\frac{b_m}{z^m}+\sum_{k=2}^{\infty}\frac{a_k}
{z^k\Big(1+\sum_{m=0}^{\infty}\frac{b_m}{z^{m+1}}\Big)^k}=z.
\end{equation}
Using the~formula
\begin{equation}\notag
\frac 1{(1+w)^k}=\sum_{s=0}^{\infty}(-1)^{s}{k-1+s\choose k-1} w^s,
\end{equation}
we finally get
\begin{equation}\notag
\sum_{m=0}^{\infty}\frac{b_m}{z^m}+\sum_{k=2}^{\infty}\frac{a_k}{z^k}
\sum_{s=0}^{\infty}(-1)^{s}{k-1+s\choose k-1}\Big(\sum_{m=0}^{\infty}
\frac{b_m}{z^{m+1}}\Big)^s=0.
\end{equation}
We obtain from this equality that $b_0=b_1=0, b_2=-a_2$ and
\begin{equation}\label{3.4*bb}
b_m+a_m+\sum_{k=2}^{m-1}a_k
\sum_{s=1}^{m-k}(-1)^{s}{k-1+s\choose k-1}\sum_{m_1+\dots+m_s=m-k-s}
b_{m_1}\dots b_{m_s}=0,\quad m=3,\dots.
\end{equation}
Moreover it is easy to deduce from (\ref{3.4*b}) and (\ref{3.4*bb}) that
\begin{equation}\label{3.4*c}
b_{2m}=-a_{2m}+O\Big(\frac 1{n^{3/2}}\Big)=-\frac{\alpha_3}{\sqrt n}
+O\Big(\frac 1{n^{3/2}}\Big)
\end{equation}
and  
\begin{align}\label{3.4*d}
b_{2m-1}&=-a_{2m-1}+2\sum_{s=1}^{m-2}sa_{2s}b_{2m-2-2s}+O\Big(\frac 1{n^{3/2}}\Big)\notag\\
&=-\frac{\alpha_4}n-\frac{(m-2)(m+1)}2\frac{\alpha_3^2}n
+O\Big(\frac 1{n^{3/2}}\Big),\quad m=2,\dots.
\end{align}
for $m=2,\dots$.
It remains to note that
$$
\frac 1{g^{(-1)}(z)}=\frac 1{z+\sum_{k=0}^{\infty}\frac{b_k}{z^k}}=
\frac 1z\Big(1-\sum_{k=0}^{\infty}\frac{b_k}{z^{k+1}}+\Big(
\sum_{k=0}^{\infty}\frac{b_k}{z^{k+1}}\Big)^2-\dots\Big)
$$
and we can write the~formal power series in $1/\sqrt n$
$$
\frac 1{g^{(-1)}(z)}=\frac 1z+\sum_{k=1}^{\infty}\frac{B_k(1/z)}{n^{k/2}}.
$$

Taking into account the~relations (\ref{3.4*c}) and (\ref{3.4*d}), 
we easily conclude that
$$
B_1(1/z)=\alpha_3\sum_{m=2}^{\infty}\frac 1{z^{2m}}=\alpha_3
\frac 1{z^3}\cdot\frac 1{z-1/z}
$$
and
\begin{equation}
B_2(1/z)=\alpha_4\sum_{m=2}^{\infty}\frac 1{z^{2m+1}}+\alpha_3^2\Big(
\sum_{m=2}^{\infty}\frac{(m-2)(m+1)}2\frac 1{z^{2m+1}}
+\frac 1{z^3}\Big(\sum_{m=1}^{\infty}\frac 1{z^{2m}}\Big)^2\Big).\notag
\end{equation}
Since
\begin{align}
\sum_{m=2}^{\infty}\frac m{z^{2m+1}}&=-\frac 12\Big(\sum_{m=2}^{\infty}\frac 1{z^{2m}}\Big)'
=-\frac 12\Big(\frac 1{z^2(z^2-1)}\Big)'
=\frac 1{z^3(z^2-1)}+\frac 1{z(z^2-1)^2},\notag\\
\sum_{m=2}^{\infty}\frac {m^2}{z^{2m+1}}&=-\frac 12\Big(\sum_{m=2}^{\infty}\frac m{z^{2m}}\Big)'
=\frac 1{z^3(z^2-1)}+\frac 1{z(z^2-1)^2}+\frac{2z}{(z^2-1)^3},\notag
\end{align}
we finally obtain
\begin{equation}
B_2(1/z)=\Big(\alpha_4-\alpha_3^2\Big)\frac 1{z^4}\cdot\frac 1{z-1/z}
+\alpha_3^2\Big(\frac 1{z^5}\cdot\frac{1}{(z-1/z)^2}+\frac 1{z^2}
\cdot\frac{1}{(z-1/z)^3}\Big).
\notag 
\end{equation}

In view of (\ref{3.4*a}) and (\ref{3.4*bb}), we see as well that $B_k(z)$ are 
functions of the~form
$$
B_k(1/z)=\sum c_{p,m}\frac 1{z^p}\frac 1{(z-1/z)^m}
$$ 
with real coefficients $c_{p,m}$ which depend on 
the~ free cumulants $\alpha_3,\dots,\alpha_{k+2}$ 
and do not depend on $n$. The~summation is carried out 
over a~finite set of non-negative integer pairs $(p,m)$. 
The~coefficients $c_{p,m}$ can be calculated explicitly in the~way described above for
the~coefficients $c_{p,m}$ of the~functions $B_1(1/z)$ and $B_2(1/z)$. 

Hence we deduce from (\ref{3.3*}) the~formal expansion
\begin{equation}\label{3.7*}
G_{\mu_n}(z)=G_{\mu_w}(z)+
\sum_{k=1}^{\infty}\frac{B_k(G_{\mu_w}(z))}{n^{k/2}}.
\end{equation}

Using integration by parts, it is not difficult to verify that
\begin{align}\label{3.8*}
B_1(G_{\mu_w}(z))=\frac {\alpha_3}{\sqrt{z^2-4}}G_{\mu_w}^3(z)
&=\frac {\alpha_3}{2\pi}\int_{-2}^2
\frac{x(x^2-3)}{\sqrt{4-x^2}}\,\frac{dx}{z-x}\notag\\
&=-\alpha_3\int_{-2}^2\frac 1{z-x}\,d\Big(\frac 13 U_2(x/2)
p_w(x)\Big),\quad z\in\mathbb C^+.
\end{align}
On the~other hand we see that if $\alpha_3\ne 0$ then 
the~function $B_2(G_{\mu_w}(z))$ is not the~Cauchy transform of some~signed measure.
Indeed, it is easy to see, using direct calculations, that
\begin{equation}\label{3.8*a}
B_2(G_{\mu_w}(z))=\frac{\alpha_3^2}{(z^2-4)^{3/2}}+g(z),\quad z\in\mathbb C^+,
\end{equation}
where the~function $g(z)$ is analytic on $\mathbb C^+$ and there exists finite limit,
for every $-\infty<t_1<t_2<+\infty$,
\begin{equation}\label{3.8*b}
\lim_{y\downarrow 0}\int_{t_1}^{t_2}\Im g(x+iy)\,dx.
\end{equation}
In addition we note that
\begin{equation}\label{3.8*c}
\lim_{y\downarrow 0}\int_{3/2}^{2}\Im \frac 1{((x+iy)^2-4)^{3/2}}\,dx=\infty.
\end{equation}

Assume now that $B_2(G_{\mu_w}(z))$ is a~Cauchy transform of 
a~real-valued function $\omega(x)$ of bounded variation on every bounded interval
and such that
\begin{equation}\notag
\int\limits_{-\infty}^{\infty}\frac{|d\omega(x)|}{1+|x|}<\infty.
\end{equation}
Then, by Stieltjes--Perron's inversion formula~\cite{Akh:1965}, we have
\begin{equation}\label{3.8*d}
\frac{\omega(t_2+0)-\omega(t_2-0)}2-
\frac{\omega(t_1+0)-\omega(t_1-0)}2=
\lim_{y\downarrow 0}\frac 1{\pi}\int_{t_1}^{t_2}\Im B_2(G_{\mu_w}(x+iy))\,dx. 
\end{equation}
Assuming in (\ref{3.8*d}) $t_1:=3/2$ and $t_2:=2$,
and taking into account (\ref{3.8*a})--(\ref{3.8*c}), we arrive at a~contradiction.

If $\alpha_3(\mu)=0$, then 
\begin{align}\label{3.9*}
B_2(G_{\mu_w}(z))=\frac {\alpha_4}{\sqrt{z^2-4}}G_{\mu_w}^4(z)
&=\frac {\alpha_4}{2\pi}\int_{-2}^2
\frac{x^4-4x^2+2}{\sqrt{4-x^2}}\,\frac{dx}{z-x}\notag\\
&=-\alpha_4\int_{-2}^2\frac 1{z-x}\,d\Big(\frac 14 U_3(x/2)
p_w(x)\Big),\quad z\in\mathbb C^+.
\end{align}

\section{ Edgeworth Expansion in the~Free CLT 
(the~case $\beta_q< \infty,\,q\ge 3$)}

In this section we prove Theorem~\ref{2.1th} and Corollary~\ref{2.1co}.
 
{\it Proof of Theorem~$\ref{2.1th}$}. 
Recall that we denote by $\mu_n$ the~distribution of $Y_n$ in (\ref{2.1}) for 
the~free random variables $X_j$. Our first step is to reduce the problem to the~case
of bounded free random variables.

\subsection {Passage to measures with bounded supports}

Let $n\in\mathcal N$. Let $\varepsilon_n\in(0,10^{-1/2}]$ be a~point at which
infimum of the~function $g_{qn2}(\varepsilon)$ from (\ref{2.4**}) is attained.
This means that
\begin{equation}\notag
\eta_{q2}(n):=\varepsilon_n^{4-q_2}+\frac{\rho_{q_2}(\mu,\varepsilon_n\sqrt n)}{\beta_{q_2}}\varepsilon_n^{-q_2}.
\end{equation}
Without loss of generality we assume that
\begin{equation}\label{4.0aa}
\eta_{q2}(n) L_{q_2n}+L_{3n}<c_1,
\end{equation}
where  
$c_1>0$ is a~sufficiently small absolute constant. By Lyapunov's inequality $\beta_3\ge m_2^{3/2}=1$, 
we obtain from (\ref{4.0aa}) that $n$ in this case has to be sufficiently
large, i.e., $n\ge c_1^{-2}\beta_3^2\ge c_1^{-2}$.

Consider free random variables 
$\tilde{X},\tilde{X}_1,\tilde{X}_2,\dots$ with distribution $\tilde{\mu}=\mathcal L(\tilde{X})$ such that 
$\tilde{\mu}([-\varepsilon_n\sqrt n,\varepsilon_n\sqrt n])=1$ and $\tilde{\mu}(B)=\mu(B)$ for all Borel sets
$B\subseteq [-\varepsilon_n\sqrt n,\varepsilon_n\sqrt n]\setminus\{0\}$. Denote by $\tilde{\mu}_n$
distribution of the~random variable 
$
\tilde{Y_n}:=(\tilde{X}_1+\dots+\tilde{X}_n)/\sqrt n.
$
In addition introduce random variables
\begin{equation}\notag
X^*:=\frac{\tilde{X}-A_n}{C_n},X_1^*:=\frac{\tilde{X_1}-A_n}{C_n},X_2^*:=\frac{\tilde{X_2}-A_n}{C_n},\dots\quad
\text{and}\quad Y_n^*:=\frac{X_1^*+\dots+X_n^*}{\sqrt n},
\end{equation}
where 
\begin{equation}\notag
A_n:=-\int_{|u|>\varepsilon_n\sqrt n}u\,\mu(du)\quad\text{and}\quad
C_n:=\Big(1-\int_{|u|>\varepsilon_n\sqrt n}u^2\,\mu(du)
-\Big(\int_{|u|>\varepsilon_n\sqrt n}u\,\mu(du)\Big)^2\Big)^{1/2}.
\end{equation}
Denote by $\mu^*$ and $\mu_n^*$
the~distributions of the~random variables $X^*$ and $Y_n^*$, respectively. We denote by $m_k^*$
and $\tilde{m}_k,\,k=0,1,\dots$, the~moments 
and by $\beta_k^*$
and $\tilde{\beta}_k,\,k=0,1,\dots$, the~absolute moments of the~distributions $\mu^*$ and  $\tilde{\mu}$, respectively.
It is obvious that $m_1^*=0$ and $m_2^*=1$.
Using (\ref{4.0aa}) we note that
\begin{equation}\label{4.*a}
|A_n|\le \varepsilon_n^{-(q_2-1)}n^{-(q_2-1)/2}\rho_{q_2}(\mu,\varepsilon_n \sqrt n)
\le\frac 1{\sqrt n}\eta_{q2}(n)L_{q_2n}  
\end{equation}
and
\begin{equation}\label{4.*b}
0\le \frac 1{C_n}-1\le 2(\rho_2(\mu,\varepsilon_n\sqrt n)+A_n^2)\le 3\eta_{q2}(n)L_{q_2n}. 
\end{equation}
By (\ref{4.0aa})--(\ref{4.*b}), we obtain
\begin{equation}\label{4.*c}
C_n^{-1}(\varepsilon_n\sqrt n+|A_n|)<\frac 13\sqrt n. 
\end{equation}
It follows from (\ref{4.*c}) that the~support of $\mu^*$ is contained in $[-\frac 13\sqrt n,\frac 13\sqrt n]$.

By (\ref{4.0aa})--(\ref{4.*b}), we easily deduce as well that
\begin{align}\label{4.*d}
&|m_3^*-m_3|\le C_n^{-3}|\tilde{m}_3-m_3|+(C_n^{-3}-1)|m_3|+C_n^{-3}(3|A_n|\tilde{m}_2+3A_n^2|\tilde{m}_1|+
|A_n|^3)\notag\\
&\le C_n^{-3}|\tilde{m}_3-m_3|+4|m_3|\eta_{q2}(n)L_{q_2n}+\frac 4{\sqrt n} \eta_{q2}(n)L_{q_2n}\notag\\
&\le C_n^{-3}\varepsilon_n^{-(q_2-3)}n^{-(q_2-3)/2}\rho_{q_2}(\mu,\varepsilon_n \sqrt n)+4\Big(|m_3|+\frac 1{\sqrt n}\Big)
\eta_{q2}(n)L_{q_2n}\le 2\sqrt n \eta_{q2}(n)L_{q_2n},
\end{align}
and, using similar arguments,
\begin{align}\label{4.*g}
\beta_3^*\le  C_n^{-3}\tilde{\beta}_3+\frac 4{\sqrt n}\eta_{q2}(n)L_{q_2n},\quad
m_4^*\le  C_n^{-4}\tilde{m}_4+5L_{3n}\eta_{q2}(n)L_{q_2n}.
\end{align}

Let $T$ be a~random variable with 
distribution $\mu_{a_n,0,0}$. Denote by $\tilde{\mu}_{a_n,0,0}$ the~distribution of $C_nT+\sqrt n A_n$.

By the~triangle inequality, we have 
\begin{equation}\label{4.0}
\Delta(\mu_n,\mu_{a_n,0,0})\le \Delta(\mu_n,\tilde{\mu}_n)+\Delta(\tilde{\mu}_n,\tilde{\mu}_{a_n,0,0})
+\Delta(\tilde{\mu}_{a_n,0,0},\mu_{a_n,0,0}).
\end{equation}
First we establish with the~help of Proposition~\ref{3.3b}
\begin{equation}\label{4.0a}
\Delta(\mu_n,\tilde{\mu}_n)\le n\Delta(\mu,\tilde{\mu})\le n\mu(\{|u|>\varepsilon_n\sqrt n\})
\le \varepsilon_n^{-q_2}n^{-(q_2-2)/2}\rho_{q_2}(\mu,\varepsilon_n \sqrt n)\le\eta_{q2}(n)L_{q_2n}.
\end{equation}

Recalling the definition of $\mu_{a_n,0,0}$ (see (\ref{2.3h}) and (\ref{2.3g}), (\ref{2.3***})), 
we note that $\mu_{a_n,0,0}$ is an~absolutely continuous measure with the~support on $[a_n-2,a_n+2]$ and its density
has the~form 
\begin{equation}\label{4.0a*}
\sqrt{4-(x-a_n)^2}/(2\pi(1+a_nx)),\quad\text{when}\quad x\in [a_n-2,a_n+2].
\end{equation}
This density does not exceed 1 on the~set $[a_n-2,a_n+2]$ and is equal to 0 outside of this set, therefore
we easily deduce the~following upper bound, using (\ref{4.*a}) and (\ref{4.*b}),
\begin{equation}\label{4.0b}
\Delta(\tilde{\mu}_{a_n,0,0},\mu_{a_n,0,0})\le c\Big(\frac 1{C_n}-1+\frac{\sqrt n A_n}{C_n}\Big)\le
c\eta_{q2}(n)L_{q_2n}. 
\end{equation}

Finally we note that $\Delta(\tilde{\mu}_n,\tilde{\mu}_{a_n,0,0})=\Delta(\mu_n^*,\mu_{a_n,0,0})$. Our next main aim is
to estimate this quantity.

By Proposition~\ref{3.1}, 
$G_{\mu_n^*}(z)=1/F_{\mu_n^*}(z)$, $z\in\mathbb C^+$, 
where $F_{\mu_n^*}(z):=F_{\mu^*}(Z(\sqrt nz))/\sqrt n$.
Here $Z(z)\in\mathcal F$ is the~solution of 
the~equation~(\ref{3.10}) with $\mu=\mu^*$. 

Consider the~functions 
\begin{align}
&S(z):=\frac 12(z+\sqrt{z^2-4}),\quad 
S_n(z):=Z(\sqrt nz)/\sqrt n,\notag\\
&S_{n1}(z):=a_n+\frac 12\Big(z-a_n+\sqrt{(z-a_n)^2-4}\Big),\qquad
z\in\mathbb C^+.\notag
\end{align}  
Note that $1/S(z)=G_{\mu_w}(z)$, where $\mu_w$ denotes the~semicircle 
measure. Since $S_n\in\Cal F$, we saw in Section~3 that there exists
$\hat{\mu}_n\in\mathcal M$ such that $1/S_n(z)=G_{\hat{\mu}_n}(z)$. In addition,  
it is easy to see, that $1/S_{n1}(z)=G_{\mu_{a_n,0,0}}(z)$. 

In order to estimate $\Delta(\mu_n^*,\mu_{a_n,0,0})$
we will apply the~Stieltjes-Perron inversion formula to the~measures
$\mu_{n}^*$ and $\mu_{a_n,0,0}$. 
For this we need further estimates for $|G_{\mu_n^*}(z)-G_{\mu_{a_n,0,0}}(z)|$ on $\mathbb C^+$.

\subsection{The~functional equation for the~function $S_n(z)$} 
Using (\ref{2.3a}) with $\mu=\mu^*$, we write, for $z\in\mathbb C^+$,
\begin{align}\label{4.1}
Z(z)G_{\mu^*}(Z(z))&=1+\frac 1{Z^2(z)}+\frac 1{Z^2(z)}
\int_{\mathbb R}\frac{u^3\,\mu^*(du)}{Z(z)-u}\notag\\
&=1+\frac 1{Z^2(z)}+\frac {m_3^*}{Z^3(z)}
+\frac 1{Z^3(z)}\int_{\mathbb R}\frac{u^4\,\mu^*(du)}{Z(z)-u}.
\end{align}
The~equation (\ref{3.10}) with $\mu=\mu^*$ may be rewritten as
\begin{equation}\label{4.2}
G_{\mu^*}(Z(z))\Big(Z(z)-z\Big)=(n-1)(1-Z(z)G_{\mu^*}(Z(z))),
\quad z\in\mathbb C^+.
\end{equation}
By (\ref{4.1}) and the~definition of $S_n(z)$, we represent (\ref{4.2}) in the~form 
\begin{equation}\label{4.3}
\Big(1+\frac 1{nS_n^2(z)}+\frac {r_{n1}(z)}{nS_n^2(z)}\Big)(S_n(z)-z)
=-\frac{n-1}n\frac 1{S_n(z)}\Big(1+\frac {m_3+r_{n2}(z)}{\sqrt n S_n(z)}\Big), 
\end{equation}
for $z\in\mathbb C^+$, where 
\begin{equation}\label{4.3*a}
r_{n1}(z):=
\int_{\mathbb R}\frac{u^3\,\mu^*(du)}{Z(\sqrt nz)-u},\qquad
r_{n2}(z):=
\int_{\mathbb R}\frac{u^4\,\mu^*(du)}{Z(\sqrt nz)-u}+m_3^*-m_3.
\end{equation}
By (\ref{4.0aa}) and (\ref{4.*c}), we obtain from Lemma~\ref{l7.5} for $\mu=\mu^*$ the~bound 
\begin{equation}\label{4.3*}
|Z(\sqrt nz)|\ge\sqrt {(n-1)/8},\quad z\in\mathbb C^+.
\end{equation}
The~functions $r_{nj}(z),\,j=1,2$, are analytic on $\mathbb C^+$ 
and with the~help of the~inequalities (\ref{4.*b})--(\ref{4.*g}) and (\ref{4.3*}) admit the~estimates, 
for $z\in\mathbb C^+$,
\begin{align}\label{4.4}
|r_{n1}(z)|&\le\int_{|u|\le \frac 13\sqrt n}\frac{|u|^3\,\mu^*(du)}{||Z(\sqrt nz)|-|u||}\le
\frac{52\beta_3^*}{\sqrt n}\le \frac{53}{\sqrt n}\Big(\tilde{\beta}_3+\frac 4{\sqrt n}\eta_{q2}(n)L_{q_2n}\Big)
\le 54L_{3n},\notag\\ 
|r_{n2}(z)|&\le\int_{|u|\le \frac 13\sqrt n}\frac{u^4\,\mu^*(du)}{||Z(\sqrt nz)|-|u||}
+|m_3^*-m_3|\le\frac {52 m_4^*}{\sqrt n}+2\sqrt n\eta_{q2}(n)L_{q_2n}\notag\\
&\le\frac {53\tilde{m}_4}{\sqrt n}+3\sqrt n\eta_{q2}(n)L_{q_2n}.
\end{align}

We deduce from (\ref{4.3}) the~following relation
\begin{equation}\label{4.5}
S_n^3(z)-zS_n^2(z)+(1+\varepsilon_{n1}(z))S_n(z)+\varepsilon_{n2}(z)=0,
\quad z\in\mathbb C^+,
\end{equation}
where  
\begin{equation}\label{4.5*}
\varepsilon_{n1}(z):=\frac 1n r_{n1}(z)
\quad\text{and}\quad 
\varepsilon_{n2}(z):=
\frac {m_3}{\sqrt n}+r_{n3}(z):=a_n+r_{n3}(z)
\end{equation} 
with 
$$
r_{n3}(z):=\Big(1-\frac 1n\Big)\frac {r_{n2}(z)}{\sqrt n}
-\frac zn\Big(1+r_{n1}(z)\Big)-\frac{m_3}{n\sqrt n}.
$$ 
\subsection{Estimates of $\varepsilon_{n1}(z)$ and $\varepsilon_{n2}(z)$}
By (\ref{4.4}), we obtain
\begin{equation}\label{4.6}
|r_{n3}(z)|\le \frac{53\tilde{m}_4}{n}+3\eta_{q2}(n)L_{q_2n}
+\frac{|z|}n\Big(1+54L_{3n}\Big)+\frac{L_{3n}}{n},
\quad z\in\mathbb C^+.
\end{equation}

Note that $\tilde{m}_4\le\beta_{q_2}(\varepsilon_n^2 n)^{(4-q_2)/2}$.
By (\ref{4.4}), (\ref{4.6}) and the last inequality,
we have, for $z\in D_1:=\{z\in\mathbb C^+:0<\Im z\le 3,|\Re z|\le 4\}$,
\begin{equation}\label{4.7}
|\varepsilon_{n1}(z)|\le 54\frac{L_{3n}}n<\frac 1{10}, 
\end{equation}
\begin{equation}\label{4.7a}
|r_{n3}(z)|\le  53\frac{\beta_{q_2}\eta_{q2}(n)}{n^{(q_2-2)/2}}+3\eta_{q2}(n)L_{q_2n}
+\frac{2|z|+L_{3n}}n
\le 56\eta_{q2}(n)L_{q_2n}+\frac{11}n
\end{equation}
and 
\begin{equation}\label{4.8}
|\varepsilon_{n2}(z)|\le 56\eta_{q2}(n)L_{q_2n}+2L_{3n}<\frac 1{10^4}.
\end{equation}

\subsection{Roots of the functional equation for $S_n(z)$}
For every fixed $z\in\mathbb C^+$,
consider the~cubic equation
$$
P(z,w):=w^3-zw^2+(1+\varepsilon_{n1}(z))w+\varepsilon_{n2}(z)=0.
$$
Denote roots of this equation by $w_j=w_j(z),\,j=1,2,3$.

We shall show that for $z\in D_1$ the~equation $P(z,w)=0$ 
has a~root, say $w_1=w_1(z)$, such that
\begin{equation}\label{4.9}
w_1=-a_n+r_{n4}(z),
\end{equation}
where 
the~quantity $r_{n4}(z)$ admits the following bound 
\begin{equation}\label{4.10}
|r_{n4}(z)|< 10^2r,\quad\text{where}\quad r:=\eta_{q2}(n)L_{q_2n}+L_{3n}^2.
\end{equation}
In addition $|w_j+a_n|\ge 10^2r,\,j=2,3$.

Indeed, introduce the~polynomials 
\begin{equation}\notag
P_1(w):=w^3-zw^2\quad\text{and}\quad P_2(w):=(1+\varepsilon_{n1}(z))w+\varepsilon_{n2}(z)
=(1+\varepsilon_{n1}(z))(w+\varepsilon_{n3}(z)), 
\end{equation}
where $\varepsilon_{n3}(z):=\varepsilon_{n2}(z)/(1+\varepsilon_{n1}(z))$.
They admit the~following estimates on the~circle $|w+a_n|=10^2r$
\begin{align}
|P_1(w)|&\le |w|^2(|w-z|)\le 2(|w+a_n|^2+a_n^2)(|w+a_n|+|z+a_n|)\notag\\
&\le 2(10^4r^2+a_n^2)(r+11/2)\le 12(10^4r^2+a_n^2) \le 24\,r
\label{4.10*a}
\end{align}
and
\begin{equation}\label{4.10*b}
|P_2(w)|=|1+\varepsilon_{n1}(z)||w+\varepsilon_{n3}(z)|\ge
(1-|\varepsilon_{n1}(z)|)|10^2r-|\varepsilon_{n3}(z)-a_n||.
\end{equation}
Since, by (\ref{4.7}), $1-|\varepsilon_{n1}(z)|\ge 9/10$ and, by (\ref{4.7})--(\ref{4.8}), 
\begin{equation}\notag
|\varepsilon_{n3}(z)-a_n|\le|r_{n3}(z)|+2|\varepsilon_{n1}(z)||\varepsilon_{n2}(z)|
\le 57\,r, 
\end{equation}
we see from (\ref{4.10*b}) that $|P_2(w)|\ge 36\,r$. This estimate and (\ref{4.10*a}) gives us the~inequality
$|P_1(w)|<|P_2(w)|$
on the~circle $|w+a_n|=10^2r$ and the~desired result follows from Rouch\'e's theorem.

\subsection{The~remaining roots $w_2$ and $w_3$ of the~equation (\ref{4.5})}
Now we shall investigate the behavior of the~roots $w=w_2(z)$ and $w=w_3(z)$.
As seen in Subsection~6.4 $w_2(z)\ne w_1(z)$ and $w_3(z)\ne w_1(z)$ for $z\in D_1$.
We shall construct a~set $D_2\subset D_1$ where $w_2(z)\ne w_3(z),\,z\in D_2$.
Since $P(z,w)=P_3(z,w)(w-w_1)$, where 
$$
P_3(z,w):=w^2-(z-w_1)w+1+\varepsilon_{n1}(z)-w_1(z-w_1),
$$ 
we see that
$w_2=w_3$ for $z\in D_1$ such that 
\begin{equation}\label{4.12}
(z-w_1)^2-4(1+\varepsilon_{n1}(z)-w_1(z-w_1))=0.
\end{equation}
We conclude from this relation that $z=\pm 2\sqrt{1+\varepsilon_{n1}(z)
+w_1^2}-w_1$. Therefore, as it is easy to see from (\ref{4.7}) and (\ref{4.9}), (\ref{4.10}), 
the relation (\ref{4.12}) does not hold for $z\in D_2$, where 
$D_2:=\{z\in\mathbb C:0<\Im z\le 3,|\Re z-a_n|\le 2-h_1\}$ and 
$h_1:=c_1^{-1/6}r$.

Hence the~roots $w_1(z),w_2(z)$ and $w_3(z)$ are distinct for $z\in D_2$.



Now we see that the~roots $w_2$ and $w_3$ have the~form
\begin{equation}\label{4.13}
w_{j}:=\frac 12\Big(z-w_1+(-1)^{j-1}\sqrt{g(z)}\Big),\quad j=2,3.
\end{equation}
where $g(z):=(z-w_1)^2-4-4\varepsilon_{n1}(z)+4w_1(z-w_1)\ne 0$ for 
$z\in D_2$. In this formula we choose the branch of the square 
satisfying $\sqrt{g(i)}\in\mathbb C^+$. 

Using (\ref{4.9}), we rewrite (\ref{4.13}) in the~following way
\begin{align}\label{4.15}
w_{j}&:=\frac 12\Big(z+a_n+(-1)^{j-1}\sqrt{(z+a_n)^2-4-4a_n(z+a_n)+r_{n5}(z)}
\Big)-\frac 12 r_{n4}(z)\notag\\
&=a_n+\frac 12\Big(z-a_n+(-1)^{j-1}\sqrt{(z-a_n)^2-4-4a_n^2+r_{n5}(z)}\Big)
-\frac 12 r_{n4}(z),
\end{align}
$j=1,2$, where 
\begin{equation}\notag
r_{n5}(z):=-4\varepsilon_{n1}(z)+(2z+6a_n-3r_{n4}(z))r_{n4}(z).
\end{equation}
From (\ref{4.7}) and (\ref{4.10}) it follows that the~following estimate holds, for $z\in D_2$,
\begin{align}\label{4.15*}
|r_{n5}(z)|&\le 4|\varepsilon_{n1}(z)|+(2|z|+6|a_n|+3|r_{n4}(z)|)|r_{n4}(z)|\notag\\
&\le 216\frac {L_{3n}}n+(10+6L_{3n}+300r)10^2r\le 1100\,r.
\end{align}
Using (\ref{4.15*}) we obtain 
\begin{equation}\label{4.15*a}
\Big|\frac{4a_n^2-r_{n5}(z)}{(z-a_n)^2-4}\Big|\le
\frac {1104\, r}{h_1}\le 1104\,c_1^{1/6}\le\frac 1{10},
\quad z\in D_2.
\end{equation}

By power series expansion of $(1+z)^{1/2},\,|z|<1$, we obtain, for $z\in D_2$,
\begin{align}
\sqrt{(z-a_n)^2-4-4a_n^2+r_{n5}(z)}=\sqrt{(z-a_n)^2-4}
+\frac{r_{n6}(z)}{\sqrt{(z-a_n)^2-4}}, \notag
\end{align}
where $|r_{n6}(z)|\le 1004 r$. By this relation, we see that, for $z\in D_2$,
\begin{align} \label{4.16}
w_j&=a_n+\frac 12\Big((z-a_n)
+(-1)^{j-1}\sqrt{(z-a_n)^2-4}\Big)\notag\\
&-\frac 12r_{n4}(z)+\frac {(-1)^{j-1}}2\frac{r_{n6}(z)}{\sqrt{(z-a_n)^2-4}},\quad j=2,3.
\end{align}

Let us show that $S_n(z)=w_3(z)$ for $z\in D_2$. By (\ref{4.0aa}), (\ref{4.9}) and (\ref{4.10}), we
see that $|w_1(z)|\le 1/6$ for $z\in D_2$.
Since $S_n(z)$ satisfies the~equation (\ref{4.5}) and,
by (\ref{4.3*}), $|S_n(z)|\ge 1/3$ for all $z\in\mathbb C^+$, we have
$S_n(z)=w_2(z)$ or $S_n(z)=w_3(z)$ for $z\in D_2$. 

First assume that, for every $z_0\in D_2$, there exists $r_0=r_0(z_0)>0$ such that $S_n(z)=w_j(z)$ 
for all $z\in D_2\cap\{|z-z_0|<r_0\}$, where $j=2$ or $j=3$. From this assumption it follows
that  $S_n(z)=w_j(z)$ for all $z\in D_2$, where $j=2$ or $j=3$. Furthermore,
it is not difficult to see that the~roots $w_2(z)$ and $w_3(z)$ 
admit the~estimates: $|w_2(z)|\le 4/3$ for $z\in D_2$, and $|w_3(z)|\ge 3/2$ for $z\in D_2$ and $\Im z\ge 2$.
The~analytic function $S_n(z)\in \mathcal F$, by (\ref{3.5**}), satisfies the~inequality $\Im S_n(z)\ge \Im z$.
Hence, under the above assumption  
$S_n(z)=w_3(z)$ for all $z\in D_2$.


Now assume that the above assumption does not hold. Then there exists a point $z_0\in D_2$ such that, 
for any $r_0>0$, there exist points
$z'\in D_2$ and $z''\in D_2$ in the disc $|z-z_0|<r_0$ such that $S_n(z')=w_2(z')$ and $S_n(z'')=w_3(z'')$. 
Let for definiteness $S_n(z_0)=w_2(z_0)$. By this assumption, there exists a sequence 
$\{z_k\}_{k=1}^{\infty}$ such that $z_k\to z_0$ and
$S_n(z_k)=w_3(z_k)$. Therefore we have $w_2(z_0)=\lim_{z_k\to z_0}w_3(z_k)$. 
Using (\ref{4.16}), rewrite this relation in the form
\begin{align} 
&a_n+\frac 12\Big(z_0-a_n
-\sqrt{(z_0-a_n)^2-4}\Big)
-\frac 12\Big(r_{n4}(z_0)+\frac{r_{n6}(z_0)}{\sqrt{(z_0-a_n)^2-4}}\Big)=\notag\\
&a_n+\frac 12\Big(z_0-a_n
+\sqrt{(z_0-a_n)^2-4}\Big)
+\frac 12\lim_{z_k\to z_0}\Big(r_{n4}(z_k)-\frac{r_{n6}(z_k)}
{\sqrt{(z_k-a_n)^2-4}}\Big).\label{4.17}
\end{align}
From (\ref{4.17}) we easily conclude with $r$ as in (\ref{4.10})
\begin{align} 
c_1^{-1/12}\sqrt{r} \le |\sqrt{(z_0-a_n)^2-4}|\le 1004 (c_1^{1/12}\sqrt{r}+r), \notag
\end{align}
a contradiction for sufficiently small $c_1>0$. Hence, the first assumption holds only and 
$S_n(z)=w_3(z),\,z\in D_2$.


Denote by $B_1$ the~set $[-2+h_1+a_n,2-h_1+a_n]$.

\subsection{Estimate of the~integral $\int\limits_{B_1}|G_{\mu_{a_n,0,0}}(x+i\varepsilon)
-G_{\mu_{n}^*}(x+i\varepsilon)|\,dx$ \,\,for $0<\varepsilon\le 1$}
We obtain an~estimate of this integral, using the~inequality
\begin{align}\label{4.25*} 
\int_{B_1}|G_{\mu_{a_n,0,0}}(x+i\varepsilon)-G_{\mu_{n}^*}(x+i\varepsilon)|\,dx&\le
\int_{B_1}|G_{\mu_{a_n,0,0}}(x+i\varepsilon)-G_{\hat{\mu}_{n}}(x+i\varepsilon)|\,dx\notag\\
&+\int_{B_1}|G_{\hat{\mu}_{n}}(x+i\varepsilon)-G_{\mu_{n}^*}(x+i\varepsilon)|\,dx.
\end{align}
Therefore we need to evaluate the~functions
\begin{equation}\label{4.25**} 
G_{\mu_{a_n,0,0}}(z)-G_{\hat{\mu}_n}(z)\quad\text{and}\quad
G_{\hat{\mu}_n}(z)-G_{\mu_n^*}(z)
\end{equation}
for $z\in D_2$.

For $z\in D_2$, using the~formula (\ref{4.15})
with $j=3$ for $S_n(z)$, we write
\begin{align}\label{4.26}
G_{\hat{\mu}_n}(z)-G_{\mu_{a_n,0,0}}(z)&=\frac 1{S_n(z)}-\frac 1{S_{n1}(z)}=\frac{S_{n1}(z)-S_n(z)}
{S_{n1}(z)S_n(z)}\notag\\
=\frac 1{2S_{n1}(z)S_n(z)}\Big(r_{n4}(z)
&+\frac{4a_n^2-r_{n5}(z)}{\sqrt{(z-a_n)^2-4}+\sqrt{(z-a_n)^2-4
-4a_n^2+r_{n5}(z)}}\Big).
\end{align}

By (\ref{4.15*a}), we have, for $z\in D_2$,
\begin{align}\label{4.27}
&|\sqrt{(z-a_n)^2-4}+\sqrt{(z-a_n)^2-4-4a_n^2+r_{n5}(z)}|\notag\\=
&|\sqrt{(z-a_n)^2-4}||1+\sqrt{1-(4a_n^2-r_{n5}(z))/((z-a_n)^2-4)}|
\ge|\sqrt{(z-a_n)^2-4}|. 
\end{align}
In addition, we see from (\ref{4.3*}) that $|S_n(z)|\ge 1/3$ for
$z\in \mathbb C^+$. The~same estimate obviously holds for $|S_{n1}(z)|$.

Therefore we can conclude from (\ref{4.26}) and (\ref{4.27}) that
\begin{align}\label{4.28}
\int_{B_1}\Big|G_{\hat{\mu}_n}(x+i\varepsilon)&-G_{\mu_{a_n,0,0}}(x+i\varepsilon)\Big|\,dx=
\int_{B_1}\Big|\frac 1{S_n(x+i\varepsilon)}-\frac 1{S_{n1}(x+i\varepsilon)}\Big|\,dx
\notag\\
&\le \frac 92\int_{B_1}\Big(|r_{n4}(x+i\varepsilon)|+\frac{4a_n^2+|r_{n5}(x+i\varepsilon)|}
{|\sqrt{(x-a_n+i\varepsilon)^2-4}|}\Big)\,dx 
\end{align}
for $0<\varepsilon\le 1$.

From (\ref{4.10}) it follows at once that
\begin{equation}\label{4.29}
\int_{B_1}|r_{n4}(x+i\varepsilon)|\,dx\le 4 \cdot 10^2r,\quad \varepsilon\in(0,1].
\end{equation}

From (\ref{4.15*}) we conclude that,
for the~same $\varepsilon$,
\begin{equation}\label{4.30}
\int_{B_1}\frac{4a_n^2+|r_{n5}(x+i\varepsilon)|}{|\sqrt{(x-a_n+i\varepsilon)^2-4}|}\,dx 
\le 1104\,r\int_{B_1}\frac{dx}{\sqrt{4-(x-a_n)^2}}\le 4416\,r.
\end{equation}

It follows from (\ref{4.28})--(\ref{4.30}) that
\begin{equation}\label{4.31}
\int_{B_1}\Big|G_{\hat{\mu}_n}(x+i\varepsilon)-G_{\mu_{a_n,0,0}}(x+i\varepsilon)\Big|\,dx
\le 3\cdot 10^4r,\quad\varepsilon\in(0,1].
\end{equation}

Now we conclude from (\ref{4.1}) that
\begin{equation}\label{4.36}
G_{\mu_n^*}(z)-G_{\hat{\mu}_n}(z) =\frac{r_{n7}(z)}{S_n(z)},\quad z\in\mathbb C^+,
\end{equation}
where
\begin{equation}\label{4.37}
r_{n7}(z):=\frac 1{nS_n^2(z)}+\frac {r_{n1}(z)}{nS_n^2(z)}.
\end{equation}
Since $|S_n(z)|\ge 1/3$ for $z\in \mathbb C^+$,
we see from (\ref{4.4}) that
\begin{equation}\label{4.38}
|r_{n7}(z)|\le \frac{9(1+54L_{3n})}{n},\qquad z\in D_2.
\end{equation}
Therefore, we deduce from (\ref{4.36}) and (\ref{4.38}) the~upper bound 
\begin{equation}\label{4.39}
\int_{B_1}|G_{\mu_n^*}(x+i\varepsilon)-G_{\hat{\mu}_n}(x+i\varepsilon)|\,dx
\le\int_{B_1}\frac {|r_{n7}(x+i\varepsilon)|}{|S_n(x+i\varepsilon)|}\,dx\le \frac {2\cdot 10^2}n, 
\quad \varepsilon\in(0,1].
\end{equation}

From (\ref{4.25*}), (\ref{4.31}) and (\ref{4.39}) we finally obtain  
\begin{equation}\label{4.39*}
\int_{B_1}|G_{\mu_{a_n,0,0}}(x+i\varepsilon)-G_{\mu_{n}^*}(x+i\varepsilon)|\,dx\le  
4\cdot10^4\,r,\quad \varepsilon\in(0,1].
\end{equation}

\subsection{Application of the~Stieltjes-Perron inversion formula}
By (\ref{4.0a*}),
we have the~relation 
\begin{equation}\label{4.40}
\int\limits_{B_1}p_{\mu_{a_n,0,0}}(x)\,dx=1-\Big(\int\limits_{[2-h_1+a_n,2+a_n]}+\int\limits_{[-2+a_n,-2+h_1+a_n]}\Big)
\frac{\sqrt{4-(x-a_n)^2}}{2\pi(1+a_nx)}\,dx\ge 1-h_1^{3/2}. 
\end{equation}
From (\ref{4.39*}) and (\ref{4.40}) we conclude, using the~Stieltjes-Perron inversion formula,
\begin{equation}\label{4.41}
\mu_n^*(B_1)\ge 1-(4\cdot 10^4+c_1^{-1/4}r^{1/2})\,r\ge 1-(4\cdot 10^4+c_1^{1/4})\,r\ge 1-(4\cdot 10^4+1)\,r.
\end{equation}
Finally we deduce from (\ref{4.39*})--(\ref{4.41}) and the~Stieltjes-Perron inversion formula
that
\begin{equation}\label{4.42}
\Delta(\mu_n^*,\mu_{a_n,0,0})\le c\,r=c\,\Big(\eta_{q2}(n) L_{q_2n}+L_{3n}^2\Big).
\end{equation}

\subsection{Completion of the~proof of Theorem~\ref{2.1th}} 
The~statement of the~theorem follows immediately from (\ref{4.0}), (\ref{4.0a}), (\ref{4.0b})
and (\ref{4.42}).
$\square$

\vskip 0,5cm
{\it Proof of Corollary~$\ref{2.1co}$}. It is easy to see that the~assertion 
of Corollary~\ref{2.1co} follows
from Theorem~\ref{2.1th} and from the~following simple formula, for $x\in \mathbb R$ and $n\in\mathbb N$,
$$
\mu_{a_n,0,0}((-\infty,x))-\mu_w((-\infty,x))=-\frac{m_3}
{3\sqrt n}(x^2-1)p_w(x)+c\theta\Big(\frac{|m_3|}{\sqrt n}\Big)^{3/2}.
$$
\vskip 0,5cm
{\it Proof of Corollary~$\ref{2.1aco}$}. The~assertion of Corollary~\ref{2.1aco} follows
immediately from (\ref{2.4}) and Proposition~\ref{3.3c}.
$\square$

\section{ Edgeworth Expansion in Free CLT
(the~case $\beta_q<\infty,\,q\ge 4$)}

In this section we prove Theorem~\ref{2.2th} and Corollary~\ref{2.2co}.
The~proof of the~theorem is similar to the~proof of
Theorem~\ref{2.1th} but with some essential technical differences.
Therefore we describe in detail those arguments which differ from
the~proof of Theorem~\ref{2.1th} and omit arguments which directly repeat
the~arguments of Section~6. We preserve all notations of Section~6.
Denote as well 
$$
S_{n2}(z):=a_n+\frac 12\Big(\big(1+b_n\big)(z-a_n)+
\sqrt{\big(1-b_n\big)^2(z-a_n)^2-4\big(1-d_n\big)}\Big),\quad z\in\mathbb C^+,
$$
where $a_n,\,b_n$ and  $d_n$ are defined in Section~2.
The~function $S_{n2}(z)\in\mathcal F$ and $1/S_{n2}(z)=G_{\mu_{a_n,b_n,d_n}}(z)$, where
$\mu_{a_n,b_n,d_n}$ is the~free Meixner measure with the~parameters
$a_n,\,b_n$ and  $d_n$, see (\ref{2.3****}).

{\it Proof of Theorem~$\ref{2.2th}$}. 
First we proceed to study 
\subsection {The~passage to measures with bounded supports}

Let $n\in\mathcal N$. Let $\varepsilon_n\in(0,10^{-1/2}]$ be a~point at which
the~infimum of the~function $g_{qn3}(\varepsilon)$ in (\ref{2.4**}) is attained.
This means that
\begin{equation}\label{5.0*}
\eta_{q3}(n):=\varepsilon_n^{5-q_3}+\frac{\rho_{q_3}(\mu,\varepsilon_n\sqrt n)}{\beta_{q_3}}\varepsilon_n^{-q_3}.
\end{equation}
Using this parameter $\varepsilon_n$, we define free random variables $\tilde{X},\tilde{X}_1,\tilde{X}_2,\dots$ and 
$X^*,X_1^*,X_2^*,\dots$ in the same way as in Section~6. We define
probability measures $\tilde{\mu}_n,\mu^*,\mu_n^*$ in the~same way as well. 

Without loss of generality we assume that
\begin{equation}\label{5.0}
\eta_{q3}(n) L_{q_3n}+L_{4n}<c_2,
\end{equation}
where   
$c_2>0$ is a~sufficiently small absolute constant. From (\ref{5.0}) it follows that
$n$ is sufficiently large $n>c_2^{-1}m_4\ge c_2^{-1}$. Here and in the~sequel we use Lyapunov's inequality
$1=m_2^{1/2}\le\beta_3^{1/3}\le m_4^{1/4}$.

Now we repeat the~arguments of Subsection~6.1.

Using (\ref{5.0}) we note that
\begin{equation}\label{5.*a}
|A_n|\le \varepsilon_n^{-(q_3-1)}n^{-(q_3-1)/2}\rho_{q_3}(\mu,\varepsilon_n \sqrt n)
\le\frac 1{\sqrt n}\eta_{q3}(n)L_{q_3n}  
\end{equation}
and
\begin{equation}\label{5.*b}
0\le \frac 1{C_n}-1\le 2(\rho_2(\mu,\varepsilon_n\sqrt n)+A_n^2)\le 3\eta_{q3}(n)L_{q_3n}. 
\end{equation}
By (\ref{5.*a}) and (\ref{5.*b}), we see that (\ref{4.*c}) holds and
the~support of $\mu^*$ is contained in $[-\frac 13\sqrt n,\frac 13\sqrt n]$.

Recalling (\ref{4.*d}) and (\ref{4.*g}), we easily deduce, by (\ref{5.0})--(\ref{5.*b}), that
\begin{align}\label{5.*d}
&|m_3^*-m_3|\le 2\sqrt n \eta_{q3}(n)L_{q_3n},\notag\\
&|m_4^*-m_4|\notag\\
&\le C_n^{-4}|\tilde{m}_4-m_4|+(C_n^{-4}-1)m_4+C_n^{-4}(4|A_n||\tilde{m}_3|+6A_n^2\tilde{m}_2
+4|A_n|^3|\tilde{m}_1|+A_n^4)\notag\\
&\le C_n^{-4}|\tilde{m}_4-m_4|+5m_4\eta_{q3}(n)L_{q_3n}+\frac {5|\tilde{m}_3|+1}{n} \eta_{q3}(n)L_{q_3n}\notag\\
&\le C_n^{-4}\varepsilon_n^{-(q_3-4)}n^{-(q_3-4)/2}\rho_{q_3}(\mu,\varepsilon_n \sqrt n)+5m_4\Big(1+\frac 2{n}\Big)
\eta_{q3}(n)L_{q_3n}\notag\\
&\le 2n \,\eta_{q3}(n)L_{q_3n}
\end{align}
and
\begin{align}\label{5.*g}
\beta_5^*\le  C_n^{-5}\tilde{\beta}_5+6L_{4n}\eta_{q3}(n)L_{q_3n}.
\end{align}
By the~triangle inequality we have 
\begin{equation}\label{5.1}
\Delta(\mu_n,\kappa_{n})\le \Delta(\mu_n,\tilde{\mu}_n)+\Delta(\tilde{\mu}_n,\tilde{\kappa}_{n})
+\Delta(\tilde{\kappa}_{n},\kappa_{n}),
\end{equation}
where, for $x\in\mathbb R$, 
\begin{align}
\tilde{\kappa}_n((-\infty,x))&:=\tilde{\mu}_{a_n,b_n,d_n}((-\infty,x))+\frac 1n\tilde{\varsigma}_n((-\infty,x))\notag\\
&:=\mu_{a_n,b_n,d_n}((-\infty,(x-\sqrt n A_n)/C_n))+\frac 1n\varsigma_n((-\infty,(x-\sqrt n A_n)/C_n)).\notag
\end{align}
Note that $\Delta(\tilde{\mu}_n,\tilde{\kappa}_{n})=\Delta(\mu_n^*,\kappa_n)$.

First we establish with the~help of Proposition~\ref{3.3b}
\begin{equation}\label{5.1a}
\Delta(\mu_n,\tilde{\mu}_n)\le n\Delta(\mu,\tilde{\mu})\le n\mu(\{|u|>\varepsilon_n\sqrt n\})
\le \varepsilon_n^{-q_3}n^{-(q_3-2)/2}\rho_{q_3}(\mu,\varepsilon_n \sqrt n)\le\eta_{q3}(n)L_{q_3n}. 
\end{equation}

We saw in Section~4 that, for $n\ge 3m_4$, $\mu_{a_n,b_n,d_n}$ is an~absolutely continuous
measure with support on the~set $B_2:=[a_n-2/e_n,a_n+2/e_n]$
and density of the~form
\begin{equation}\label{5.1aa}
p_{\mu_{a_n,b_n,d_n}}(x):=\frac {\sqrt{4(1-d_n)-(1-b_n)^2(x-a_n)^2}}{2\pi(b_nx^2+a_n(1-b_n)x+1-d_n)},
\quad x\in B_2.
\end{equation}
This density does not exceed $1$ on the~set $B_2$
and is equal $0$ outside of this set.

The~signed measure $\varsigma_n$ has density $p_{\varsigma_n}$, see (\ref{2.5a***}), 
which does not exceed $1$ by modulus on the~set $B_2$
and is equal to zero outside of $B_2$.

Therefore, in view of (\ref{5.0}), a~simple calculation shows that
\begin{align}\label{5.1b}
\Delta(\tilde{\kappa}_n,\kappa_n)&\le \Delta(\tilde{\mu}_{a_n,b_n,d_n},\mu_{a_n,b_n,d_n})
+\frac 1n \Delta(\tilde{\varsigma}_n,\varsigma_n)
\le c\Big(\frac 1{C_n}-1+\frac{\sqrt n A_n}{C_n}\Big)\notag\\
&\le c\varepsilon_n^{-(q_3-1)}n^{-(q_3-2)/2}\rho_{q_3}(\mu,\varepsilon_n \sqrt n)\le c\eta_{q3}(n)L_{q_3n}.
\end{align}

Our next aim is to estimate the~quantity $\Delta(\mu_n^*,\kappa_n)$.
In order to estimate $\Delta(\mu_n^*,\kappa_{n})$ 
we need to apply the~inversion formula to $\mu_{n}^*$ and $\kappa_{n}$. 
We shall now derive the~necessary estimates for $|G_{\mu_n^*}(z)-G_{\kappa_{n}}(z)|$ on $\mathbb C^+$.

\subsection{The~functional equation for $S_n(z)$.}
Using (\ref{2.3a}) with $\mu=\mu^*$, we write, for $z\in\mathbb C^+$,
\begin{equation}\label{5.2}
Z(z)G_{\mu^*}(Z(z))=1+\frac 1{Z^2(z)}+\frac {m_3^*}{Z^3(z)}
+\frac {m_4^*}{Z^4(z)}+\frac 1{Z^4(z)}
\int_{\mathbb R}\frac{u^5\,\mu^*(du)}{Z(z)-u}.
\end{equation}
By (\ref{5.2}) and the~definition of $S_n(z)$, the~equation (\ref{3.10}) with $\mu=\mu^*$ 
may be rewritten as
\begin{align}\label{5.3}
&\Big(1+\frac 1{nS_n^2(z)}+\frac{m_3^*}{n^{3/2}S_n^3(z)}+
\frac {m_4^*+\zeta_{n1}(z)}{n^2S_n^4(z)}\Big)(S_n(z)-z)\notag\\
&=-\frac{n-1}n\frac 1{S_n(z)}\Big(1+\frac{m_3^*}{\sqrt n S_n(z)}
+\frac {m_4^*+\zeta_{n1}(z)}{n S_n^2(z)}\Big) 
\end{align}
for $z\in\mathbb C^+$, where 
$
\zeta_{n1}(z):=\int_{\mathbb R}\frac{u^5\,\mu^*(du)}{Z(\sqrt nz)-u}.
$

We deduce from (\ref{5.3}) the~following relation, for $z\in\mathbb C^+$,
\begin{equation}\label{5.5}
S_n^5(z)-zS_n^4(z)+S_n^3(z)
+\frac{\zeta_{n2}(z)}{\sqrt n}S_n^2(z)
+\frac{\zeta_{n3}(z)}n S_n(z)-\frac{\zeta_{n4}(z)z}{n^2} =0,
\end{equation}
where $\zeta_{n2}(z):=m_3^*-z/\sqrt n$, $\zeta_{n3}(z)(z):=m_4^*+\zeta_{n1}(z)-zm_3^*/\sqrt n$ and
$\zeta_{n4}(z)(z):=m_4^*+\zeta_{n1}(z)$.
Note that the~functions $\zeta_{nj}(z),j=1,2,3,4$, are analytic on $\mathbb C^+$.

\subsection{Estimates of the~functions $\zeta_{nj}(z),\,j=1,2,3$ on the~set $D_1$.}  
From (\ref{5.0}) and (\ref{4.*c}) we deduce, using Lemma~\ref{l7.5}, that (\ref{4.3*}) holds.
Therefore, in view of (\ref{5.*b}) and (\ref{5.*g}), we arrive at the~estimate
\begin{align}\label{5.4}
|\zeta_{n1}(z)|&\le\int_{|u|\le\frac 13\sqrt n}\frac{|u|^5\,\mu^*(du)}{||Z(\sqrt nz)|-|u||}\le
\frac{52\beta_5^*}{\sqrt n}\le 53\,\frac{\tilde{\beta}_5}{\sqrt n}+312\,\frac{L_{4,n}}{\sqrt n}\eta_{q3}(n)L_{q_3n}\notag\\
&\le 53\,\frac{\beta_{q_3}(\varepsilon_n^2 n)^{(5-q_3)/2}}{\sqrt n}+312\,\frac{L_{4n}}{\sqrt n}\eta_{q3}(n)L_{q_3n}
\le 54\,n\,\eta_{q3}(n)L_{q_3n},
\quad z\in\mathbb C^+.
\end{align}

For $z\in D_1$, by (\ref{5.*d}) and (\ref{5.4}), we get the~bounds
\begin{align}\label{5.4a}
&\frac{|\zeta_{n2}(z)|}{\sqrt n}\le 2(L_{3n}+\eta_{q3}(n)L_{q_3n}),\notag\\
&\frac{|\zeta_{n3}(z)|}{n}\le\frac{m_4^*+|\zeta_{n1}(z)|}{n}+\frac{5|m_3^*|}{n^{3/2}} \le 2L_{4n}+
56\,\eta_{q3}(n) L_{q_3n}\notag\\
&\frac{|\zeta_{n4}(z)|}{n}\le\frac{m_4^*+|\zeta_{n1}(z)|}{n}\le L_{4n}+56\,\eta_{q3}(n) L_{q_3n}.
\end{align}

\subsection{The~roots of the functional equation (\ref{5.5}) for $S_n(z)$}
For every fixed $z\in\mathbb C^+$
consider the~equation
\begin{equation}\label{5.4c}
Q(z,w):=w^5-zw^4+w^3+\frac{\zeta_{n2}(z)}{\sqrt n} w^2
+\frac{\zeta_{n3}(z)}n w-\frac{\zeta_{n4}(z)z}{n^2} =0.
\end{equation}


Denote the~roots of the equation (\ref{5.4c}) by $w_j=w_j(z),\,j=1,\dots,5$.
Let us show that for every fixed $z\in D_1$ 
the~equation $Q(z,w)=0$ 
has three roots, say $w_j=w_j(z),\,j=1,2,3$, such that
\begin{equation}\label{5.6}
|w_j|<r':=15(L_{4n}+\eta_{q3}(n) L_{q_3n})^{1/2},
\quad j=1,2,3,
\end{equation}
two roots, say $w_j,\,j=4,5$, such that $|w_j|\ge r'$ for $j=4,5$.
Recalling (\ref{5.0}) 
we see that 
the~bound $r'<\frac 1{100}$ holds.

Consider the~polynomials
$$
Q_1(z,w):=w^5-zw^4+\frac {\zeta_2(z)}{\sqrt n} w^2
+\frac{\zeta_{n3}(z)}n w-\frac{\zeta_{n4}(z) z}{n^2} \quad\text{and}\quad
Q_2(w):=w^3.
$$ 
The~following estimates hold on the~circle $|w|=r'$ for $z\in D_1$: $|w|^5=(r')^2|w|^3
\le 10^{-4}|w|^3$, and $|zw^4|=|z|r'|w|^3\le\frac 1{20}|w|^3$.
Since $n\,r'\ge 15\sqrt{m_4}\ge 15m_2=15$ and, by Proposition~\ref{3.3c}, $r'\ge 15L_{3n}$,  
we have as well, using (\ref{5.4a}),
\begin{align}
&\frac{|\zeta_{n2}(z)|}{\sqrt n}|w^2|=\frac{|\zeta_{n2}(z)|}{\sqrt n}\frac 1{r'}|w|^3
\le\frac{|w|^3}5,\notag\\ 
&\frac{|\zeta_{n3}(z)|}{n}|w|=\frac{|\zeta_{n3}(z)|}{n}\frac 1{(r')^2}
|w|^3\le\frac{4|w|^3}{15},\notag\\
&\frac{|\zeta_{n4}(z)|}{n^2}|z|=\frac{|\zeta_{n4}(z)|}{n^2}
\frac {|z|}{(r')^3}|w|^3\le \frac{4|z|}{15nr'}|w|^3\le\frac{4|w|^3}{45}.\notag
\end{align}
We see from the~last five inequalities that
$|Q_1(z,w)|\le \frac 89|Q_2(w)|$ on the~circle $|w|=r'$.  
Therefore, by Rouch\'e's theorem, we obtain that the~polynomial
$Q_1(z,w)+Q_2(w)$ has only three roots which are less than $r'$ in modulus, 
as claimed.  


Represent $Q(z,w)$ in the~form
$$
Q(z,w)=(w^2+s_1w+s_2)(w^3+g_1w^2+g_2w+g_3),
$$
where $w^3+g_1w^2+g_2w+g_3=(w-w_1)(w-w_2)(w-w_3)$.
From this formula we derive the~relations
\begin{align}\label{5.7}
s_1+g_1=-z,\quad &s_2+s_1g_1+g_2=1,\quad s_2g_1+s_1g_2+g_3=
\frac{\zeta_{n2}(z)}{\sqrt n},
\notag \\
&s_2g_2+s_1g_3=\frac {\zeta_{n3}(z)}n,\quad s_2g_3=-\frac{\zeta_{n4}(z)z}{n^2}.
\end{align}
By Vieta's formulae and (\ref{5.6}), note that
\begin{equation}\label{5.8a}
|g_1|\le 3r',\quad |g_2|\le 3(r')^2,\quad
|g_3|\le (r')^3.
\end{equation}
Now we obtain from (\ref{5.7}) and (\ref{5.8a}) the~following bounds,
for $z\in D_1$,
\begin{equation}\label{5.8}
|s_1|\le 5+3r',\quad |1-s_2|\le 3r'(4r'+5)\le 16r'\le\frac 12.
\end{equation}
Then we conclude from (\ref{5.*d}), (\ref{5.4a}), (\ref{5.7})--(\ref{5.8}) that, for the~same $z$,
\begin{align}\label{5.9}
\Big|g_2-\frac{\zeta_{n4}(z)}{n}\Big| &\le \Big|g_2-\frac{\zeta_{n3}(z)}{n}\Big|+\frac{|m_3^*||z|}{n^{3/2}}
\le \frac{|s_1|}{|s_2|}|g_3|+\frac{|s_2-1|}{|s_2|}\frac{|\zeta_{n3}(z)|}n+(r')^3\notag\\
&\le 11(r')^3+8(r')^3+(r')^3=20(r')^3.
\end{align}

Denote $a_n^*:=m_3^*/\sqrt n$, $\rho_n^*:=L_{4n}^*-1/n:=(m_4^*-1)/n$ and $\rho_n:=(m_4-1)/n$ .
By (\ref{5.*d}), it is easy to see that
\begin{align}
|a_n-a_n^*|=\frac{|m_3-m_3^*|}{\sqrt n}
\le 2\eta_{q3}(n)L_{q_3n}\,\,\text{and}\,\,
|\rho_n-\rho_n^*|=\frac{|m_4-m_4^*|}n
\le 2\eta_{q3}(n)L_{q_3n}.\label{5.9*}
\end{align}

From the~first three relations in (\ref{5.7}) it follows that
\begin{equation}\label{5.9a}
g_1+zg_1^2=a_n+\rho_n z+\zeta_{n5}(z),
\end{equation}
where 
$$
\zeta_{n5}(z):=\frac{\zeta_{n1}(z)z}{n}
+\Big(g_2-\frac{\zeta_{n4}(z)}n\Big)z-g_1^3+2g_1g_2-g_3-(a_n-a_n^*)-(\rho_n-\rho_n^*)z.
$$
By (\ref{5.4}), (\ref{5.8a}), (\ref{5.9}) and (\ref{5.9*}), we get  the~following estimate, for $z\in D_1$,
\begin{align}\label{5.9b}
|\zeta_{n5}(z)|&\le \frac{|\zeta_{n1}(z)z|}{n}+\Big|g_2-\frac{\zeta_{n4}(z)}n\Big||z|
+|g_1^3|+2|g_1g_2|+|g_3|+|a_n-a_n^*|+|\rho_n-\rho_n^*|\notag\\
&\le 274\eta_{q3}(n)L_{q_3n}+146(r')^3\le 8\cdot 10^5\Big(\eta_{q3}(n)L_{q_3n}+L_{4n}^{3/2}\Big).
\end{align}
Rewrite (\ref{5.9a}) in the~form
$$
g_1(1+a_n z)=a_n+\rho_n z+(a_n+\rho_nz)\Big(\frac 1{1+g_1z}-1\Big)+a_n g_1z
+\frac{\zeta_{n5}(z)}{1+g_1 z}.
$$
Taking into account (\ref{5.8a}), (\ref{5.9b}) and Proposition~\ref{3.3c}  
this relation leads us to the~bound, for $z\in D_1$,
\begin{align}\label{5.10}
&|g_1-a_n-(\rho_n-a_n^2)z|\notag\\
&\le \frac{|a_n^3 z^2|}{|1+a_n z|}+\frac{|a_n\rho_n z^2|}{|1+a_n z|}
+\frac{|a_n g_1^2z^2|}{|1+a_n z||1+g_1z|}
+\frac{|\rho_n g_1z^2|}{|1+a_n z||1+g_1z|}+
\frac{|\zeta_{n5}(z)|}{|1+a_n z||1+g_1 z|}\notag\\
&\le 50L_{3n}^3+50L_{3n}L_{4n}+500L_{3n}(r')^2
+150L_{4n}r'+16\cdot 10^5\Big(\eta_{q3}(n)L_{q_3n}+L_{4n}^{3/2}\Big)\notag\\
&\le 18\cdot 10^5\Big(\eta_{q3}(n)L_{q_3n}+L_{4n}^{3/2}\Big).
\end{align}

To find the~roots $w_4$ and $w_5$, we need to solve the~equation
$w^2+s_1w+s_2=0$. Using (\ref{5.7}), we have, 
for $j=4,5$,
\begin{align}\label{5.11}
w_j&=\frac 12\Big(-s_1+(-1)^j\sqrt{s_1^2-4s_2}\Big)\notag\\&=\frac 12\Big(
z+g_1+(-1)^j\sqrt{(z+g_1)^2-4(1+(z+g_1)g_1-g_2)}\Big)\notag\\
&=\frac 12\Big(z+g_1+(-1)^j\sqrt{(z-g_1)^2-4-4(g_1^2-g_2)}\Big)
=\frac 12 \zeta_{n6}(z)+a_n+\notag\\
&+\frac 12\Big(\big(1+b_n\big)(z-a_n)
+(-1)^j\sqrt{\big(1-b_n\big)^2(z-a_n)^2-4\big(1-d_n\big)
+\zeta_{n7}(z)}\Big),
\end{align}
where 
\begin{align}\label{5.11a}
\zeta_{n6}(z)&:=g_1-a_n-b_n(z-a_n),\notag\\
\zeta_{n7}(z)&:=
-3\zeta_{n6}^2(z)-2\zeta_{n6}(z)(4a_n+(1+3b_n)(z-a_n))\notag\\
&+4(g_2-L_{4n})-4b_n(z-a_n)(2a_n+b_n(z-a_n)).
\end{align}
We choose the~branch of the~analytic square root according to the~condition $\Im w_4(i)\ge 0$.
Note that the roots $w_4(z)$ and $w_5(z)$ are continuous functions in $D_1$. 

\subsection{Estimates of the~functions $\zeta_{n6}(z)$ and $\zeta_{n7}(z)$ on the~set $D_1$.} 
We obtain, by (\ref{5.10}), 
\begin{equation}\label{5.11*}
|\zeta_{n6}(z)|\le 18\cdot 10^5\Big(\eta_{q3}(n)L_{q_3n}+L_{4n}^{3/2}\Big)+|a_nb_n|
\le (18\cdot 10^5+1)\Big(\eta_{q3}(n)L_{q_3n}+L_{4n}^{3/2}\Big).
\end{equation}
By (\ref{5.4}), (\ref{5.9}) and (\ref{5.9*}), we have
\begin{equation}\label{5.12}
|g_2-L_{4n}|\le 56\eta_{q3}(n)L_{q_3n}+20(r')^3\le 10^5(\eta_{q3}(n)L_{q_3n}+L_{4n}^{3/2}).
\end{equation}
Then, using (\ref{5.0}), (\ref{5.11*}) and (\ref{5.12}), we easily deduce from (\ref{5.11a})
\begin{align}\label{5.13}
|\zeta_{n7}(z)|\le &3|\zeta_{n6}(z)|^2+2|\zeta_{n6}(z)|(4|a_n|+(1+3b_n)(|z|+|a_n|))+4|g_2-L_{4n}|\notag\\
&+4b_n(|z|+|a_n|)(2|a_n|+b_n(|z|+|a_n|))\le
3\cdot 10^7\Big(\eta_{q3}(n)L_{q_3n}+L_{4n}^{3/2}\Big).
\end{align}

\subsection{The~roots $w_4$ and $w_5$}
We saw in Subsection~7.4 that $w_4(z)\ne w_j(z),\,z\in D_1$, for $j=1,2,3$. 
Returning to (\ref{5.11}), it follows from (\ref{5.13}) that $w_4(z)\ne w_5(z)$ for $z\in D_3$, where
\begin{equation}\notag
D_3:=\Big\{z\in\mathbb C:0<\Im z\le 3,|\Re z-a_n|\le \frac 2{e_n}
-h_2\Big\}
\end{equation}
where $e_n:=(1-b_n)/\sqrt{1-d_n}$ and $h_2:=c_2^{-1/6}\Big(\eta_{q3}(n)L_{q_3n}+L_{4n}^{3/2}\Big)$. 

Since the~constant $c_2>0$ is sufficiently small, we have, by (\ref{5.13}), for $z\in D_3$,
\begin{equation}\label{5.14}
|\zeta_{n7}(z)|/|((1-b_n)^2(z-a_n)^2-4(1-d_n)|\le 4\cdot 10^7c_2^{1/6}<10^{-2}. 
\end{equation}
Therefore, using power series expansion of $(1+z)^{1/2},\,|z|<1$, we obtain, for the~same $z$,
\begin{align}
\sqrt{(1-b_n)^2(z-a_n)^2-4(1-d_n)+\zeta_{n7}(z)}&=\sqrt{(1-b_n)^2(z-a_n)^2-4(1-d_n)}\notag\\
&+\frac{\zeta_{n8}(z)}{\sqrt{(1-b_n)^2(z-a_n)^2-4(1-d_n)}}, \notag
\end{align}
where $|\zeta_{n8}(z)|\le 4\cdot 10^7\Big(\eta_{q3}(n)L_{q_3n}+L_{4n}^{3/2}\Big)$. By this relation, we see that
\begin{align} \label{5.15}
w_j(z)&=a_n+\frac 12\Big(\big(1+b_n\big)(z-a_n)
+(-1)^j\sqrt{\big(1-b_n\big)^2(z-a_n)^2-4\big(1-d_n\big)}\Big)\notag\\
&+\frac 12\zeta_{n6}(z)+\frac {(-1)^j}2\frac{\zeta_{n8}(z)}{\sqrt{(1-b_n)^2(z-a_n)^2-4(1-d_n)}},\quad j=4,5,
\end{align}
for $z\in D_3$.

Let us show that $S_n(z)=w_4(z)$ for $z\in D_3$. By (\ref{5.0}) and (\ref{5.6}), we see that
$|w_j(z)|\le 1/6$ for $z\in D_3$.
Since, in view of (\ref{4.3*}), $|S_n(z)|\ge 1/3$ for all $z\in\mathbb C^+$,
we have $S_n(z)=w_4(z)$ or $S_n(z)=w_5(z)$ for $z\in D_3$. 

Assume that, for every $z_0\in D_3$, there exists $r_0=r_0(z_0)>0$ such that $S_n(z)=w_j(z)$ 
for all $z\in D_3\cap\{|z-z_0|<r\}$, where $j=4$ or $j=5$. From this assumption it follows
that  $S_n(z)=w_j(z)$ for all $z\in D_3$, where $j=4$ or $j=5$.
By (\ref{3.5**}), we have $|S_n(2i)|>1$.
In addition it follows from (\ref{5.6}) and (\ref{5.11}), (\ref{5.11*}), (\ref{5.13})
that $|w_4(2i)|>1$ and $|w_j(2i)|<1,\,j=1,2,3,5$. Hence in this case $S_n(z)=w_4(z)$ for $z\in D_3$.

If the above assumption is not true, there exists a point $z_0\in D_3$ such that, for any $r_0>0$, there exist points
$z'\in D_3$ and $z''\in D_3$ from the disc $|z-z_0|<r_0$ such that $S_n(z')=w_4(z')$ and $S_n(z'')=w_5(z'')$. 
Let for definiteness $S_n(z_0)=w_5(z_0)$. By assumption there exists a sequence 
$\{z_k\}_{k=1}^{\infty}$ such that $z_k\to z_0$ 
and $S_n(z_k)=w_4(z_k)$. Therefore we have $w_5(z_0)=\lim_{z_k\to z_0}w_4(z_k)$. 
Rewrite this relation, using (\ref{5.15}),
\begin{align} \label{5.16}
&a_n+\frac 12\Big(\big(1+b_n\big)(z_0-a_n)
-\sqrt{\big(1-b_n\big)^2(z_0-a_n)^2-4\big(1-d_n\big)}\Big)\notag\\
&+\frac 12\Big(\zeta_{n6}(z_0)-\frac{\zeta_{n8}(z_0)}{\sqrt{(1-b_n)^2(z_0-a_n)^2-4(1-d_n)}}\Big)=\notag\\
&a_n+\frac 12\Big(\big(1+b_n\big)(z_0-a_n)
+\sqrt{\big(1-b_n\big)^2(z_0-a_n)^2-4\big(1-d_n\big)}\Big)\notag\\
&+\frac 12\lim_{z_k\to z_0}\Big(\zeta_{n6}(z_k)+\frac{\zeta_{n8}(z_k)}
{\sqrt{(1-b_n)^2(z_k-a_n)^2-4(1-d_n)}}\Big).
\end{align}
From (\ref{5.16}) we easily conclude
\begin{align} 
c_2^{-1/12}\sqrt{\eta_{q3}(n)L_{q_3n}+L_{4n}^{3/2}} 
&\le |\sqrt{\big(1-b_n\big)^2(z_0-a_n)^2-4\big(1-d_n\big)}|\notag\\
&\le 4\cdot 10^7\sqrt{\eta_{q3}(n)L_{q_3n}+L_{4n}^{3/2}}\Big(c_2^{1/12}
+\sqrt{\eta_{q3}(n)L_{q_3n}+L_{4n}^{3/2}}\Big), \notag
\end{align}
which leads to a contradiction for sufficiently small $c_2>0$. Hence our assumption holds and 
$S_n(z)=w_4(z),\,z\in D_3$.


Denote by $B_3$ the~set $\Big[-\frac 2{e_n}+h_2+a_n,\frac 2{e_n}-h_2+a_n\Big]$.

\subsection{Estimate of the~integral $\int_{B_3}\big|G_{\mu_{n}^*}(x+i\varepsilon)-G_{\mu_{a_n,b_n,d_n}}(x+i\varepsilon)
-\frac 1n (G_{\mu_{a_n,b_n,d_n}}(x+i\varepsilon))^3\big|\,dx$ for $0<\varepsilon\le 1$}

We obtain an~estimate of this integral, using the~inequality
\begin{align}\label{5.17*} 
&\int_{B_3}\big|G_{\mu_{n}^*}(x+i\varepsilon)-G_{\mu_{a_n,b_n,d_n}}(x+i\varepsilon)
-\frac 1n (G_{\mu_{a_n,b_n,d_n}}(x+i\varepsilon))^3\big|\,dx\notag\\
&\le\int_{B_3}|G_{\mu_{a_n,b_n,d_n}}(x+i\varepsilon)-G_{\hat{\mu}_{n}}(x+i\varepsilon)|\,dx+\notag\\
&+\int_{B_3}\big|G_{\hat{\mu}_{n}}(x+i\varepsilon)-G_{\mu_{n}^*}(x+i\varepsilon)
-\frac 1n (G_{\mu_{a_n,b_n,d_n}}(x+i\varepsilon))^3\big|\,dx.
\end{align}
Therefore we need to evaluate the~functions $G_{\mu_{a_n,b_n,d_n}}(z)-G_{\hat{\mu}_n}(z)$
and $G_{\hat{\mu}_n}(z)-G_{\mu_n^*}(z)-\frac 1n (G_{\mu_{a_n,b_n,d_n}}(z))^3$ for $z\in D_{3}$.

For $z\in D_3$, using the~formula (\ref{5.11})
with $j=4$ for $S_n(z)$, we write
\begin{align}\label{5.18}
&{S_{n2}(z)S_n(z)}\Big(\frac 1{S_n(z)}-\frac 1{S_{n2}(z)}\Big)
=S_{n2}(z)-S_n(z)=-\frac 12 \zeta_{n6}(z)\notag\\
&-\frac 12\frac{\zeta_{n7}(z)}{\sqrt{(1-b_n)^2(z-a_n)^2-4(1-d_n)}
+\sqrt{(1-b_n)^2(z-a_n)^2-4(1-d_n)+\zeta_{n7}(z)}}.
\end{align}

Using (\ref{5.14}) we get, for $z\in D_3$,
\begin{align}\label{5.19}
&\Big|\sqrt{(1-b_n)^2(z-a_n)^2-4(1-d_n)}
+\sqrt{(1-b_n)^2(z-a_n)^2-4(1-d_n)+\zeta_{n7}(z)}\Big|\notag\\=
&\Big|\sqrt{(1-b_n)^2(z-a_n)^2-4(1-d_n)}\Big|\Big|1+\sqrt{1+\zeta_{n7}(z)/
((1-b_n)^2(z-a_n)^2-4(1-d_n))}\Big|\notag\\
&\ge|\sqrt{(1-b_n)^2(z-a_n)^2-4(1-d_n)}|. 
\end{align}
It is easy to see that the bound $|S_n(z)|\ge 1/3,\,z\in\mathbb C^+$,
holds for $S_{n2}(z)$ as well.
Therefore in the~same way as in the~proof of (\ref{4.29}) and (\ref{4.30})
we conclude from (\ref{5.11*}), (\ref{5.13}) and (\ref{5.18}) that
\begin{align}\label{5.20}
\int_{B_3}\Big|G_{\hat{\mu}_n}(&x+i\varepsilon)-G_{\mu_{a_n,b_n,d_n}}(x+i\varepsilon)\Big|\,dx=
\int_{B_3}\Big|\frac 1{S_n(x+i\varepsilon)}-\frac 1{S_{n2}(x+i\varepsilon)}\Big|\,dx\notag\\
&\le \frac 92\int_{B_2}\Big(|\zeta_{n6}(x+i\varepsilon)|+\frac{|\zeta_{n7}(x+i\varepsilon)|}
{|\sqrt{(1-b_n)^2(x-a_n+i\varepsilon)^2-4(1-d_n)}|}\Big)\,dx \notag\\
&\le c\,\Big(\eta_{q3}(n)L_{q_3n}+L_{4n}^{3/2}\Big).
\end{align}

Now we deduce from (\ref{4.1}) with the~probability measure $\mu^*$, involving $\varepsilon_n$
introduced in (\ref{5.0*}), the~relation
\begin{equation}\label{5.23}
G_{\mu_n^*}(z)-G_{\hat{\mu}_n}(z)-\frac 1{nS_n^3(z)}
=\frac {r_{n1}(z)}{nS_n^3(z)}, \quad z\in \mathbb C^+,
\end{equation}
where the~function $r_{n1}(z)$ is defined in (\ref{4.3*a}).
Since (\ref{4.7}) holds for $r_{n1}(z)$, we see that
\begin{equation}\label{5.24}
\frac {|r_{n1}(z)|}{n|S_n^3(z)|}\le 1458\,\frac {L_{3n}}{n},
\qquad z\in D_3. 
\end{equation}
Since, for the~same $z$,
\begin{equation}\label{5.24*}
\Big|\frac 1{S_n^3(z)}-\frac 1{S_{n2}^3(z)}\Big|
\le 2\Big|\frac 1{S_n(z)}-\frac 1{S_{n2}(z)}\Big|\Big(\frac 1{|S_n(z)|^2}+
\frac 1{|S_{n2}(z)|^2}\Big)\le 36\Big|\frac 1{S_n(z)}-\frac 1{S_{n2}(z)}\Big|,
\end{equation}
we obtain, using (\ref{5.20})--(\ref{5.24*}) and Proposition~\ref{3.3c}, that,
for $0<\varepsilon\le 1$,
\begin{align}\label{5.25}
&\int_{B_3}\Big|G_{\mu_n^*}(x+i\varepsilon)-G_{\hat{\mu}_n}(x+i\varepsilon)
-\frac 1n (G_{\mu_{a_n,b_n,d_n}}(x+i\varepsilon))^3\Big|\,dx\notag\\
&\le\int_{B_3}\Big|G_{\mu_n^*}(x+i\varepsilon)-G_{\hat{\mu}_n}(x+i\varepsilon)
-\frac 1n (G_{\hat{\mu}_{n}}(x+i\varepsilon))^3\Big|\,dx\notag\\
&+\frac 1n\int\limits_{B_3}\Big|(G_{\mu_{a_n,b_n,d_n}}(x+i\varepsilon))^3
-(G_{\hat{\mu}_{n}}(x+i\varepsilon))^3\Big|\,dx\le  
c\,\Big(\eta_{q3}(n)L_{q_3n}+L_{4n}^{3/2}\Big). 
\end{align}

From (\ref{5.17*}), (\ref{5.20}) and (\ref{5.25}) we finally get, 
for $0<\varepsilon\le 1$,
\begin{align}\label{5.25*}
\int_{B_3}\big|G_{\mu_{n}^*}(x+i\varepsilon)-G_{\mu_{a_n,b_n,d_n}}(x+i\varepsilon)
&-\frac 1n (G_{\mu_{a_n,b_n,d_n}}(x+i\varepsilon))^3\big|\,dx\notag\\
&\le c\,\Big(\eta_{q3}(n)L_{q_3n}+L_{4n}^{3/2}\Big). 
\end{align}

\subsection{Application of the~Stieltjes-Perron inversion formula}

Using (\ref{5.1aa}), we have the~relation
\begin{equation}\label{5.26}
\int_{B_3}p_{\mu_{a_n,b_n,d_n}}(x)\,dx\ge 1-h_2^{3/2}.  
\end{equation}
where $p_{\mu_{a_n,b_n,d_n}}(x)$ denotes the~density $p_{\mu_{a_n,b_n,d_n}}(x)$ of the~measure $\mu_{a_n,b_n,d_n}$ 
It is not difficult to verify that
$$
(G_{\mu_{a_n,b_n,d_n}}(z))^3=\int\limits_{\mathbb R}\frac{\varsigma_{n1}(dx)}{z-x}
=\int_{\mathbb R}\frac{p_{\varsigma_{n1}}(x)\,dx}{z-x},
\quad z\in\mathbb C^+,
$$
where
\begin{align}
p_{\varsigma_{n1}}(x)&:=\frac 1{8\pi}
\sqrt{(4(1-d_{n})-(1-b_n)^2(x-a_n)^2)_+}\notag\\
&\times\frac{3((1+b_n)x+(1-b_n)a_n)^2+(1-b_n)^2(x-a_n)^2-
4(1-d_n)}{(b_nx^2+(1-b_n)a_n x+1-d_{n})^3}\notag
\end{align}
for $x\in B_2$ and $p_{\varsigma_{n1}}(x)=0$ for $x\notin B_2$.
Therefore we easily deduce the~obvious upper bounds
\begin{equation}\label{5.27}
\Big|\int_{B_3}p_{\varsigma_{n1}}(x)\,dx\Big|\le ch_2^{3/2},\quad
\int_{\mathbb R\setminus B_3}|p_{\varsigma_{n1}}(x)|\,dx\le ch_2^{3/2}\,\,\text{and}\,\,
\Delta(\varsigma_{n1},\varsigma_n)\le c\,\Big(|a_n|+L_{4n}\Big).
\end{equation}
From (\ref{5.25*})--(\ref{5.27}) and the~Stieltjes-Perron inversion formula we conclude that
\begin{equation}\label{5.28}
\mu_n^*(B_3)\ge 1-c\,\Big(\eta_{q3}(n)L_{q_3n}+L_{4n}^{3/2}\Big). 
\end{equation}
We finally conclude from (\ref{5.25*}), (\ref{5.27}), (\ref{5.28}) and the~Stieltjes-Perron inversion formula that
\begin{equation}\label{5.29}
\Delta(\mu_n^*,\kappa_{n})\le c\,\Big(\eta_{q3}(n)L_{q_3n}+L_{4n}^{3/2}\Big).
\end{equation}

\subsection{Completion of the~proof of Theorem~\ref{2.2th}}
The~statement of the~theorem follows immediately from (\ref{5.1}), (\ref{5.1a}), (\ref{5.1b}),
(\ref{5.29}) and Proposition~\ref{3.3c}.
$\square$
\vskip 0,5cm
{\it Proof of Corollary~$\ref{2.2co}$}.
Recalling the~definition of the~density $p_{\mu_{a_n,b_n,d_n}}(x)$ of the~measure $\mu_{a_n,b_n,d_n}$ 
for $n\ge c_2^{-1}m_4$, we see that
\begin{align}
&p_{\mu_{a_n,b_n,d_n}}(x+a_n)\notag\\
&=\frac 1{2\pi}\sqrt{(4(1-d_n)-(1-b_n)^2 x^2)_+}
(1+d_n-b_n-a_nx-(b_n-a_n^2)(x^2-1))\notag\\
&+c\theta(L_{4n}+a_n^2)^{3/2},\qquad x\in\mathbb R.\notag
\end{align}
In addition we have, for $x\in \mathbb R$,
\begin{align}
&\frac 1{2\pi}\int_{-\infty}^x\sqrt{(4(1-d_n)-(1-b_n)^2 u^2)_+}\,du\notag \\
&=(1-d_n+b_n)\mu_w((-\infty,x))
+\big(\frac 12 d_n-b_n\big)x\frac 1{2\pi}\sqrt{(4-x^2)_+}+c\,\theta L_{4n}^{3/2}\notag
\end{align}
and, for $x\in(-l_n,l_n)$, where $l_n:=\max\{2,2/e_n\}$,
\begin{align}
\big|\sqrt{(4(1-d_n)-(1-b_n)^2 x^2)_+}&-\sqrt{(4-x^2)_+}\big|\notag\\
&\le \frac{cL_{4n}}
{\sqrt{(4(1-d_n)-(1-b_n)^2 x^2)_+}+\sqrt{(4-x^2)_+}} .\notag
\end{align}
Using these formulae and the following obvious relations
\begin{equation}\notag
\int x\sqrt{4-x^2}\,dx=-\frac 13(4-x^2)^{3/2}\quad\text{and}\quad
\int (x^2-1)\sqrt{4-x^2}\,dx=-\frac 14x(4-x^2)^{3/2},
\end{equation}
we obtain from (\ref{2.6*}), using some simple calculations, the~representation (\ref{2.7*}).
$\square$

\section{Proof of Theorem~\ref{th7}} 

In this section we keep the notations of Sections~6, 7.
By definition of the random variable $X^*$ (see Subsection~7.1 in the case $q_3=5$) we note
that $\mu^*=\mu$ for $n>10L^2$. Hence for these $n$ we have $\mu^*_n=\mu_n$ as well.
By Lemma~\ref{l7.4}, $S_n(z)$ is continuous up to the real axis and, by Lemma~\ref{l7.5},
$|S_n(z)|\ge 1.03/3$ for $z\in \mathbb C^+\cup \mathbb R$ and for $n\ge c(\mu)$. 
Since 
\begin{equation}\notag
G_{\mu_n}(z)=\int_{\mathbb R}\frac{\mu(du)}{S_n(z)-u/\sqrt n},\qquad z\in\mathbb C^+, 
\end{equation}
where $supp (\mu)\subseteq [-L,L]$, we conclude that 
$G_{\mu_n}(z)$ is a continuous up to the real axis. By Lemmas~\ref{l7.4}, \ref{l7.5}
and (\ref{3.4}),
$\hat{\mu}_n$ and $\mu_n$ are absolutely continuous measures with continuous densities
$\hat{p}_n(x)$ and $p_n(x)$, respectively, defined by
\begin{equation}\notag
\hat{p}_n(x)=-\lim_{\varepsilon\downarrow 0}\frac 1{\pi}\Im \frac 1{S_n(x+i\varepsilon)} \quad\text{and}
\quad p_n(x)=-\lim_{\varepsilon\downarrow 0}\frac 1{\pi}\Im G_{\mu_n}(x+i\varepsilon).
\end{equation}
In addition, $\hat{p}_n(x)\le 1$ and $p_n(x)\le 2$ for all $x\in \mathbb R$ and $n\ge c_1(\mu)$.  
We will use the estimate, for $x\in\mathbb R$,   
\begin{equation}\label{loc.1}
|p_n(x)-p_{\mu_{a_n,b_n,d_n}}(x)-\frac 1n p_{\varsigma_{n1}}(x)|\le |\hat{p}_n(x)-p_{\mu_{a_n,b_n,d_n}}x)|
+|p_n(x)-\hat{p}_n(x)-\frac 1n p_{\varsigma_{n1}}(x)|. 
\end{equation}
Repeating the arguments of Subsection~7.7, we easily obtain the upper bounds, for $x\in B_3$
and $n\ge c(\mu)$,
\begin{equation}\label{loc.2}
|\hat{p}_n(x)-p_{\mu_{a_n,b_n,d_n}}x)|\le c\frac {L_{5n}+L_{4n}^{3/2}}{\sqrt{4(1-d_n)
-(1-b_n)^2(x-a_n)^2}} 
\end{equation}
and
\begin{equation}\label{loc.3}
|p_n(x)-\hat{p}_n(x)-\frac 1np_{\varsigma_{n1}}(x)|\le c\Big(\frac{L_{3n}}n
+\frac {L_{5n}+L_{4n}^{3/2}}{n\sqrt{4(1-d_n)
-(1-b_n)^2(x-a_n)^2}}\Big). 
\end{equation}

On the other hand, by the formula (\ref{5.1aa}), we have
\begin{equation}\label{loc.4}
p_{\mu_{a_n,b_n,d_n}}(x)=\Big(1+\frac 12 d_n-a_n^2-a_n(x-a_n)-(b_n-a_n^2)(x-a_n)^2\Big)
p_w(e_n(x-a_n))
+c\theta L_{4n}^{3/2} 
\end{equation}
for $x\in B_2$, where $e_n:=(1-b_n)/\sqrt{1-d_n}$. Returning to Subsection~7.8 we note that
\begin{equation}\label{loc.5}
p_{\varsigma_{n1}}(x)=((x-a_n)^2-1)p_w(e_n(x-a_n))
+c\theta (|a_n|+L_{4n})  
\end{equation}
for $x\in B_2$.

Applying (\ref{loc.2})--(\ref{loc.5}) to (\ref{loc.1}) we arrive at the relation
\begin{align}\label{loc.6}
p_n(x)=v_n(x-a_n)
+c\theta\frac {L_{5n}+L_{4n}^{3/2}}{\sqrt{4(1-d_n)
-(1-b_n)^2(x-a_n)^2}},\quad x\in B_3, 
\end{align}
where
\begin{align}\notag
v_n(x):=\Big(1+\frac 12 d_n-a_n^2-\frac 1{n}-a_nx-\Big(b_n-a_n^2-\frac 1{n}\Big)x^2\Big)
p_w(e_nx),\quad x\in\mathbb R. 
\end{align}
The statement (\ref{asden1}) of the theorem easily follows from (\ref{loc.6}). 
Thus, Theorem~\ref{th7} is completely proved.

\section{Asymptotic expansion of $\int_{\mathbb R} |p_n(x)-p_w(x)|\,dx$ as $n\to\infty$} 

In this section we shall prove Corollary~\ref{corth7.1}. Note that for free
bounded random variables the inequality (\ref{5.28}) has the form $\mu_n(B_3)\ge 1-c(L_{5n}+L_{4n}^{3/2})$.
Hence, 
we have 
\begin{align}\label{exp1.1}
\int_{\mathbb R} |p_n(x)-p_w(x)|\,dx=\int_{B_3} |p_n(x)-p_w(x)|\,dx+\theta\big(c|a_n|^{3/2}
+\frac {c(\mu)}{n^{3/2}}\big).
\end{align}
First let us assume that $m_3\ne 0$. Using (\ref{loc.6}), we easily conclude that
\begin{align}\notag
&\int_{B_3} |p_n(x)-p_w(x)|\,dx= \int_{[-2,2]} |(1-a_nx)p_w(x)-p_w(x+a_n)|\,dx
+\theta\big(c|a_n|^{3/2}+\frac {c(\mu)}n\big)\notag\\
&=\frac{|a_n|}{2\pi}\int_{[-2,2]} |x|\frac{|3-x^2|}{\sqrt{4-x^2}}\,dx+\theta\big(c|a_n|^{3/2}+\frac {c(\mu)}n\big)=
\frac{2|a_n|}{\pi}+\theta\big(c|a_n|^{3/2}+\frac {c(\mu)}n\big).\notag
\end{align}
Applying this relation to (\ref{exp1.1}) we arrive at the expansion (\ref{2.9}).

Let $m_3=0$. By (\ref{loc.6}), we easily get
\begin{align}
&\qquad\qquad\qquad\qquad\int_{B_3} |p_n(x)-p_w(x)|\,dx\notag\\
&=\int_{[-2,2]}\Big|p_w(x)-p_w\Big(\frac x{e_n}\Big)+\Big(\frac 12 d_n-\frac 1n-\Big(b_n
-\frac 1n\Big)\frac{x^2}{e_n^2}\Big)p_w(x)\Big|\,\frac{dx}{e_n}+\theta\frac {c(\mu)}{n^{3/2}}\notag\\
&=\frac 1{2\pi}\int_{[-2,2]}\frac{|2\Big(d_n-\frac 2n\Big)-\Big(4\Big(b_n
-\frac 1n\Big)+\frac 12 d_n-\frac 1n-b_n+\frac 12d_n\Big)x^2+\Big(b_n-\frac 1n\Big)x^4|}{\sqrt{4-x^2}}\,dx\notag\\
&+\theta\frac {c(\mu)}{n^{3/2}}=\frac{|m_4-2|}{2\pi n}\int_{[-2,2]}\frac{|2-4x^2+x^4|}{\sqrt{4-x^2}}\,dx
+\theta\frac {c(\mu)}{n^{3/2}}=\frac{2|m_4-2|}{\pi n}+\theta\frac {c(\mu)}{n^{3/2}}.\notag
\end{align}
Thus, (\ref{2.10}) is proved.

\section{Asymptotic expansion of the free entropy and the free Fisher information} 

In this section we prove Corollaries~\ref{corth7.2} and \ref{corth7.3}. First we find an asymptotic expansion of 
the logarithmic energy $E(\mu_n)$ of the measure $\mu_n$. Recall that (see~\cite{HiPe:2000})
\begin{align}\label{exp.1}
-E(\mu_n)&=\int\int_{\mathbb R^2}\log|x-y|\,\mu_n(dx)\mu_n(dy)=I_1(\mu_n)+I_2(\mu_n):=\notag\\
&=\int\int_{B_3\times B_3}\log|x-y|\,\mu_n(dx)\mu_n(dy)+
\int\int_{\mathbb R^2\setminus (B_3\times B_3)}\log|x-y|\,\mu_n(dx)\mu_n(dy).
\end{align}
Using (\ref{asden2})  
and the estimates 
$p_n(x)\le 2,\,|x|\le 2+L/\sqrt n$ and $p_n(x)=0,\,|x|\ge 2+L/\sqrt n$,
we conclude that
\begin{align}\label{exp.2}
|I_2(\mu_n)|&\le 4\int_{\mathbb R\setminus B_3}p_n(x)\,\int_{\mathbb R}|\log|x-y||p_n(y)\,dy\,dx\le 
c(\mu)\int_{\mathbb R\setminus B_3}p_n(x)\,dx\notag\\
&\le c(\mu)n^{-3/2}.
\end{align}
By (\ref{loc.6}), we have 
\begin{align}\label{exp.3}
I_1(\mu_n)&= \int\int_{B_3\times B_3}\log|x-y|\,v_n(x-a_n)v_n(y-a_n)\,dx\,dy
+c(\mu)\theta(L_{5n}+L_{4n}^{3/2})\notag\\
&=\int\int_{\mathbb R^2}\log|x-y|\,v_n(x)v_n(y)\,dx\,dy
+c(\mu)\theta n^{-3/2}.
\end{align}
Recalling the definition of $v_n(x)$ we see that
\begin{align}\label{exp.4}
&\int\int_{\mathbb R^2}\log|x-y|\,v_n(x)v_n(y)\,dx\,dy=\tilde{I}_1(v_n)+\tilde{I}_2(v_n)+\tilde{I}_3(v_n)
+\tilde{I}_4(v_n)+r_n\notag\\
&:=\Big(1+\frac 12 d_n-a_n^2-\frac 1{n}\Big)^2 \int\int_{\mathbb R^2}\log|x-y|\,
p_w(e_nx)\,p_w(e_ny)\,dx\,dy\notag\\
&-2\Big(1+\frac 12 d_n-a_n^2-\frac 1{n}\Big)a_n\int\int_{\mathbb R^2}x\log|x-y|\,
p_w(e_nx)\,p_w(e_ny)\,dx\,dy\notag\\
&+a_n^2\int\int_{\mathbb R^2}xy\log|x-y|\,p_w(e_nx)\,
p_w(e_ny)\,dx\,dy\notag\\
&-2\Big(b_n-a_n^2-\frac 1{n}\Big)\int\int_{\mathbb R^2}x^2\log|x-y|\,
p_w(e_nx)\,p_w(e_ny)\,dx\,dy
+c(\mu)\theta n^{-3/2}. 
\end{align}

In view of $E(\mu_w)=1/4$ and $e_n^{-2}=1-d_n+2b_n+c\theta L_{4n}^2,\,
\log e_n=\frac{d_n}2-b_n+c\theta L_{4n}^2$,
note that
\begin{align}\label{exp.5}
\tilde{I}_1(v_n)&=\Big(1+\frac 12 d_n-a_n^2-\frac 1{n}\Big)^2 e_n^{-2}
\Big(\int\int_{\mathbb R^2}\log|x-y|
p_w(x)p_w(y)\,dx\,dy-\log e_n\Big)\notag\\
&=-E(\mu_w)+\frac{a_n^2}2+c(\mu)\theta n^{-2}. 
\end{align}
Since the function $\int_{\mathbb R}\log|x-y|p_w(y)\,dy$ is even, we see that
$\tilde{I}_2(v_n)=0$. In order to calculate $\tilde{I}_3(v_n)$ we easily deduce that
\begin{equation}\notag
\int_{-2}^2 up_w(u)\log|x-u|\,du=-x+x^3/6,\quad x\in[-2,2]. 
\end{equation}
Therefore we obtain
\begin{equation}\label{exp.6}
\tilde{I}_3(v_n)=-a_n^2\int_{-2}^2xp_w(x)(x-x^3/6)\,dx+c(\mu)\theta n^{-2}=-\frac 23a_n^2+c(\mu)\theta n^{-2}. 
\end{equation}

Using the following well-known formula (see~\cite{HiPe:2000}, p. 197)
\begin{equation}\notag
\int_{-2}^2 p_w(u)\log|x-u|\,du=\frac{x^2}4-\frac 12,\quad x\in[-2,2],
\end{equation}
we deduce
\begin{equation}\notag
\int\int_{\mathbb R^2}x^2\log|x-y|\,
p_w(x)\,p_w(y)\,dx\,dy=\int_{-2}^2x^2\Big(\frac{x^2}4-\frac 12\Big)p_w(x)\,dx=0. 
\end{equation}
Therefore
\begin{equation}\label{exp.7}
\tilde{I}_4(v_n)=c(\mu)\theta n^{-2}. 
\end{equation}

By (\ref{exp.3})--(\ref{exp.7}), we arrive at the formula
\begin{equation}\notag
I_1(\mu_n)=-E(\mu_w)-\frac 16a_n^2+c(\mu)\theta n^{-3/2}. 
\end{equation}
Finally, in view of (\ref{exp.1})--(\ref{exp.2}), we get
\begin{equation}\notag
-E(\mu_n)=-E(\mu_w)-\frac 16a_n^2+c(\mu)\theta n^{-3/2}  
\end{equation}
and we obtain the assertion of Corollary~\ref{corth7.2}.

Now let us prove Corollary~\ref{corth7.3}. We shall show that the free Fisher information of the measure $\mu_n$ 
has the form (\ref{2.11}).
Denote 
\begin{align}\label{exp.9}
\Phi(\mu_n)=\Phi_1(\mu_n)+\Phi_2(\mu_n):=\frac{4\pi^2}3\int_{B_3}p_n(x)^3\,dx+
\frac{4\pi^2}3\int_{\mathbb R\setminus B_3}p_n(x)^3\,dx.
\end{align}
As before we see that, by (\ref{asden2}),
\begin{equation}\label{exp.10}
\Phi_2(\mu_n)\le c(\mu)\int_{\mathbb R\setminus B_3}p_n(x)\,dx\le c(\mu)n^{-3/2}.
\end{equation}
On the other hand, by (\ref{loc.6}), we have 
\begin{equation}\label{exp.11}
\Phi_1(\mu_n)=\frac{4\pi^2}3\int_{\mathbb R}v_n(x)^3\,dx+c(\mu)\theta n^{-3/2}. 
\end{equation}
It is easy to see that the integral on the right hand-side of (\ref{exp.11}) is equal to
\begin{align}
\Big(1&+3\Big(\frac 12 d_n-a_n^2-\frac 1n\Big)\Big)e_n^{-1}\int_{\mathbb R}p_w(x)^3\,dx
-3\Big(b_n-a_n^2-\frac 1n\Big)e_n^{-3}\int_{\mathbb R}x^2p_w(x)^3\,dx \notag\\
&+3a_n^2e_n^{-3}\int_{\mathbb R}x^2p_w(x)^3\,dx+c(\mu)\theta L_{4n}^{3/2}=\frac 3{4\pi^2}(1+a_n^2)
+c(\mu)\theta n^{-3/2}.\notag
\end{align}
Therefore we finally have by (\ref{exp.9})--(\ref{exp.11})
\begin{equation}\notag
\Phi(\mu_n)=1+a_n^2+c(\mu)\theta n^{-3/2}=\Phi(\mu_w)+a_n^2
+c(\mu)\theta n^{-3/2}. 
\end{equation}
Thus, Corollary~\ref{corth7.3} is proved.

\section{Appendix 1. Proof of Theorem~\ref{2.1*th}}

In this Appendix we keep the~notations of Section~6. 

\subsection {Passage to measures with bounded supports}
Let $n\in\mathcal N$. Let $\varepsilon_n\in(0,10^{-1/2}]$ be a~point at which
the~infimum of the~function $g_{qn1}(\varepsilon)$ from (\ref{2.4**}) is attained.
This means that
\begin{equation}\notag
\eta_{q1}(n):=\varepsilon_n^{3-q_1}+\frac{\rho_{q_1}(\mu,\varepsilon_n\sqrt n)}{\beta_{q_1}}\varepsilon_n^{-q_1}.
\end{equation}
Using this parameter $\varepsilon_n$, we define free random variables $\tilde{X},\tilde{X}_1,\tilde{X}_2,\dots$ and 
$X^*,X_1^*,X_2^*,\dots$ in the same way as in Section~5. We define
probability measures $\tilde{\mu}_n,\tilde{\mu}_{w}=\tilde{\mu}_{0,0,0},\mu^*,\mu_n^*$ in the~same way as well. 

Without loss of generality we assume that
\begin{equation}\label{6.0aa}
\eta_{q1}(n) L_{q_1n}+1/n<c_3,
\end{equation}
where  
$c_3>0$ is a~sufficiently small absolute constant.

Using (\ref{6.0aa}) we note that
\begin{equation}\label{6.*a}
|A_n|\le \varepsilon_n^{-(q_1-1)}n^{-(q_1-1)/2}\rho_{q_1}(\mu,\varepsilon_n \sqrt n)
\le\frac 1{\sqrt n}\eta_{q1}(n)L_{q_1n}  
\end{equation}
and
\begin{equation}\label{6.*b}
0\le \frac 1{C_n}-1\le 2\Big(\rho_2(\mu,\varepsilon_n\sqrt n)+A_n^2\Big)\le 3\eta_{q1}(n)L_{q_1n}. 
\end{equation}
By (\ref{6.0aa})--(\ref{6.*b}), we obtain that (\ref{4.*c}) holds and 
the~support of $\mu^*$ is contained in $[-\frac 13\sqrt n,\frac 13\sqrt n]$.
By (\ref{6.0aa})--(\ref{6.*b}), we easily deduce as well that
\begin{equation}\label{6.*c}
\beta_3^*\le C_n^{-3}\tilde{\beta}_3+\frac 4{\sqrt n}\eta_{q1}(n)L_{q_1n}. 
\end{equation}

By the~triangle inequality, we have 
\begin{equation}\label{6.0}
\Delta(\mu_n,\mu_w)\le \Delta(\mu_n,\tilde{\mu}_n)+\Delta(\tilde{\mu}_n,\tilde{\mu}_{w})+\Delta(\tilde{\mu}_w,\mu_{w}).
\end{equation}

Furthermore, we have the~following inequalities 
\begin{align}\label{6.0a}
\Delta(\mu_n,\tilde{\mu}_n)
&\le \varepsilon_n^{-q_1}n^{-(q_1-2)/2}\rho_{q_1}(\mu,\varepsilon_n \sqrt n)\le \eta_{q1}(n)L_{q_1n},\notag\\
\Delta(\tilde{\mu}_{w},\mu_{w})&\le
c\varepsilon_n^{-(q_1-1)}n^{-(q_1-2)/2}\rho_{q_1}(\mu,\varepsilon_n \sqrt n)\le c\eta_{q1}(n)L_{q_1n}.
\end{align}
Our next aim is to estimate $\Delta(\tilde{\mu}_n,\tilde{\mu}_{w})=\Delta(\mu_n^*,\mu_{w})$.

As in Section~6, let $Z(z)\in\mathcal F$ be the~solution of 
the~equation~(\ref{3.10}) with $\mu=\mu^*$. Denote $S_n(z):=Z(\sqrt n z)/\sqrt n$.

\subsection{The~functional equation for $S_n(z)$}
Using the~formula
\begin{equation}\label{6.1}
Z(z)G_{\mu^*}(Z(z))=1+\frac 1{Z^2(z)}+\frac 1{Z^2(z)}
\int_{\mathbb R}\frac{u^3\,\mu^*(du)}{Z(z)-u},
\end{equation}
and the~equation (\ref{3.10}) with $\mu=\mu^*$ we arrive at the~following functional equation for $S_n(z)$
\begin{equation}\label{6.2}
S_n^3(z)-zS_n^2(z)+(1+r_n(z))S_n(z)-(1+r_n(z))\frac zn=0,
\quad z\in\mathbb C^+, 
\end{equation}
where $r_n(z):=\int_{\mathbb R}\frac{u^3\,\mu^*(du)}{Z(\sqrt n z)-u}$. 
From (\ref{6.0aa}) and (\ref{4.*c}) we deduce, using Lemma~\ref{l7.5}, that (\ref{4.3*}) holds.
Therefore, in view of (\ref{6.*c}), we obtain
\begin{equation}\label{6.3}
|r_n(z)|\le \frac{52\beta_3^*}{\sqrt n}\le 53\,\frac{\tilde{\beta}_3}{\sqrt n}
+\frac{208}n \eta_{q1}(n)L_{q_1n}\le 54\,\eta_{q1}(n)L_{q_1n}<\frac 1{10}
\end{equation}
for $z\in\mathbb C^+$. Note that we obtain the~functional equation (\ref{6.2}) from (\ref{4.5})
replacing  $\varepsilon_{n1}(z)$ by $r_n(z)$ and $\varepsilon_{n2}(z)$ by $-(1+r_n(z))z/n$.

\subsection{The roots of the functional equation for $S_n(z)$}
For every fixed $z\in\mathbb C^+$
consider the~cubic equation
$$
P(z,w):=w^3-zw^2+(1+r_n(z))w-(1+r_n(z))\frac zn=0.
$$
As in Section~6 denote the~roots of this equation by $w_j=w_j(z),\,j=1,2,3$.
Repeating the~arguments of Subsection~6.4 we prove that
\begin{equation}\label{6.4}
w_1=\frac zn+\hat{r}_n(z),\quad\text{where}\quad |\hat{r}_n(z)|<\frac{10^3}{n^2},
\end{equation}
and $|w_j-z/n|\ge 10^3/n^2,\,j=2,3$, for $z\in D_1$. Hence $w_1\ne w_j$ for $j=2,3$ and $z\in D_1$.

As in Subsection~6.5 
we obtain that $w_2\ne w_3$ for $z\in D_4:=\{z\in\mathbb C:0<\Im z\le 3,|\Re z|\le 2-h_3\}$, where
$h_3:=c_3^{-1/6}(\eta_{q1}(n)L_{q_1n}+1/n)$. 
Hence the~roots $w_1(z),w_2(z)$ and $w_3(z)$ 
are distinct for $z\in D_4$. Moreover $w_1(z)$ 
satisfies, by (\ref{6.0aa}) and (\ref{6.4}), the~inequality 
\begin{equation}\label{6.6}
|w_1(z)|\le 6/n,\quad z\in D_4.
\end{equation}
Using the~arguments of
Subsection~6.5 we deduce the~formula~(\ref{4.13}) for $z\in D_4$ where
$g(z):=(z-w_1)^2-4-4r_n(z)+4w_1(z-w_1)\ne 0,\,z\in D_4$. 
Then we rewrite the~formula~(\ref{4.13}) as follows 
\begin{equation}\label{6.7}
w_{j}:=\frac 12\Big(z+(-1)^{j-1}\sqrt{z^2-4+\tilde{r}_{n}(z)}\Big)-\frac 12 w_1(z),\quad j=1,2,
\end{equation}
where $\tilde{r}_n(z):=2zw_1(z)-3w_1^2(z)-4r_n(z)$. By (\ref{6.3}) and (\ref{6.6}),
this function admits the~bound, for $z\in D_4$,
\begin{equation}\label{6.8}
|\tilde{r}_n(z)|\le 10|w_1(z)|+3|w_1(z)|^2+4|r_n(z)|\le 280\Big(\eta_{q1}(n)L_{q_1n}+1/n\Big).
\end{equation}
In the same way as in Subsection~6.5 we obtain that $S_n(z)=w_3(z),\,z\in D_4$.
Denote $B_4:=[-2+h_3,2-h_3]$.

\subsection{Estimate of the~integral $\int_{B_4}|G_{\mu_{w}}(x+i\varepsilon)
-G_{\mu_{n}^*}(x+i\varepsilon)|\,dx$ for $0<\varepsilon\le 1$}

We obtain an~estimate of this integral, using the~inequality
\begin{align}\label{6.9} 
\int_{B_4}|G_{\mu_{w}}(x+i\varepsilon)-G_{\mu_{n}^*}(x+i\varepsilon)|\,dx&\le
\int_{B_4}|G_{\mu_{w}}(x+i\varepsilon)-G_{\hat{\mu}_{n}}(x+i\varepsilon)|\,dx+\notag\\
&\int_{B_4}|G_{\hat{\mu}_{n}}(x+i\varepsilon)-G_{\mu_{n}^*}(x+i\varepsilon)|\,dx.
\end{align}
Evaluating the~function $G_{\mu_{w}}(z)-G_{\hat{\mu}_n}(z)$ for $z\in D_4$ in the~same way as in 
Subsection~6.6, we arrive at the~bound 
\begin{equation}\label{6.10} 
\int_{B_4}|G_{\mu_{w}}(x+i\varepsilon)-G_{\hat{\mu}_{n}}(x+i\varepsilon)|\,dx\le
c\,\Big(\eta_{q1}(n)L_{q_1n}+1/n\Big).
\end{equation}

Now we conclude from (\ref{6.1}) that
\begin{equation}\label{6.12}
G_{\mu_n^*}(z)-G_{\hat{\mu}_n}(z) =\frac{1+r_n(z)}{nS_n^3(z)},\quad z\in\mathbb C^+.
\end{equation}
Since $|S_n(z)|\ge 1/3$ for $z\in \mathbb C^+$,
we see from (\ref{6.3}) and (\ref{6.12}) that
\begin{equation}\label{6.13}
\int_{B_4}|G_{\mu_n^*}(x+i\varepsilon)-G_{\hat{\mu}_n}(x+i\varepsilon)|\,dx
\le\int_{B_4}\frac {1+|r_{n}(x+i\varepsilon)|}{n|S_n(x+i\varepsilon)|^3}\,dx\le \frac {120}n, 
\quad \varepsilon\in(0,1].
\end{equation}

From (\ref{6.9}), (\ref{6.10}) and (\ref{6.13}) we finally obtain 
\begin{equation}\label{6.14}
\int_{B_4}|G_{\mu_{w}}(x+i\varepsilon)-G_{\mu_{n}^*}(x+i\varepsilon)|\,dx\le  
c\,\Big(\eta_{q1}(n) L_{q_1n}+1/n\Big),\quad \varepsilon\in(0,1].
\end{equation}

\subsection{Completion of the~proof of Theorem~\ref{2.1*th}}
Note that
\begin{equation}\label{6.15}
\int_{B_4}p_{\mu_{w}}(x)\,dx\ge 1-h_3^{3/2}. 
\end{equation}
From (\ref{6.14}) and (\ref{6.15}) we deduce, using the~Stieltjes-Perron inversion formula,
\begin{equation}\label{6.16}
\mu_n^*(B_4)\ge 1-c\,\Big(\eta_{q1}(n) L_{q_1n}+1/n\Big). 
\end{equation}
Finally we deduce from (\ref{6.14})-- (\ref{6.16}) and the~Stieltjes-Perron inversion formula
that
\begin{equation}\label{6.17}
\Delta(\mu_n^*,\mu_{w})\le c\,\Big(\eta_{q1}(n) L_{q_1n}+1/n\Big).
\end{equation}

The~statement of the~theorem follows immediately from
(\ref{6.0}), (\ref{6.0a}) and (\ref{6.17}).
$\square$
\vskip 0,5cm
{\it Proof of Corollary~$\ref{2.1*co}$}.
The~inequality~(\ref{2.4a**}) follows immediately from (\ref{2.4a*}) and the~Lyapunov inequality
$1=m_2^{1/2}\le \beta_q^{1/q}$ for $q\ge 2$.
$\square$



\end{document}